\documentclass[a4paper,10pt]{amsart}
\usepackage{amsmath,amsthm,amssymb,enumerate}
\usepackage{epsfig}
\usepackage[dvipsnames]{xcolor}
\usepackage{esint}
\usepackage{bbm}
\usepackage{amssymb}
\usepackage{amsmath}
\usepackage{amssymb}
\usepackage{amsmath,amsthm}
\usepackage{graphicx}
\usepackage{tikz}
\usepackage{tikz-cd}
\usetikzlibrary{calc}
\usepackage{amsfonts}
\usepackage{amsxtra}
\usepackage{euscript,mathrsfs}
\usepackage[left=4cm,right=4cm,top=4cm,bottom=3.5cm]{geometry}
\usepackage[colorlinks=true,citecolor=MidnightBlue,urlcolor=MidnightBlue,
linkcolor=BrickRed]{hyperref}
\allowdisplaybreaks

\usepackage{enumitem}
\setenumerate{label={\rm (\alph{*})}}

\newcommand{\R}{\mathbb R}
\newcommand{\N}{\mathbb N}

\newcommand{\E}{\operatorname{E}\!}

\renewcommand{\leq}{\leqslant}

\newcommand{\dif}{\mathrm{d}}

\DeclareMathOperator{\dist}{dist}

\renewcommand{\dif}{\operatorname{d}\!}
\newcommand{\lebe}{\operatorname{L}}
\newcommand{\sobo}{\operatorname{W}}

\newcommand{\imag}{\operatorname{i}}
\newcommand{\locc}{\operatorname{loc}}

\newcommand{\hold}{\operatorname{C}}

\newcommand{\bv}{\operatorname{BV}}

\newcommand{\ball}{{B}}
\newcommand{\di}{\operatorname{div}}
\newcommand{\A}{\mathbb{A}}

\newcommand{\bd}{\operatorname{BD}}

\newcommand{\spt}{\operatorname{spt}}
\newcommand{\lin}{\operatorname{Lin}}
\newcommand{\id}{\operatorname{Id}}

\newcommand{\trace}{\operatorname{Tr}}
\newcommand{\I}{\mathcal{I}}

\usepackage{esint}

\newcommand{\sym}{\operatorname{sym}}
\newcommand{\dashint}{\fint}

\newcommand{\bl}{\operatorname{BL}}

\newcommand{\C}{\mathbb{C}}

\newcommand{\trli}{\mathrm{F}}
\renewcommand{\bl}{\operatorname{BL}}
\newcommand{\image}{\mathrm{im\,}}

\newcommand{\mres}{\mathbin{\vrule height 1.6ex depth 0pt width
0.13ex\vrule height 0.13ex depth 0pt width 1.3ex}}

\allowdisplaybreaks

\numberwithin{equation}{section}

\newtheorem{theorem}{Theorem}[section]
\newtheorem{lemma}[theorem]{Lemma}
\newtheorem{proposition}[theorem]{Proposition}

\theoremstyle{definition}

\theoremstyle{remark}

\newtheorem{remark}[theorem]{Remark}
\newtheorem{example}[theorem]{Example}
\newtheorem{conjecture}[theorem]{Open Problem}

\numberwithin{equation}{section}

\newcommand{\zeroset}{\setminus\{0\}}
\begin{document}

%\begin{frontmatter}

\title[Limiting Trace Inequalities for Differential Operators]{On Limiting Trace Inequalities for \\ Vectorial Differential Operators}
\author[F. Gmeineder]{Franz Gmeineder}
\author[B. Rai\c{t}\u{a}]{Bogdan Rai\c{t}\u{a}}
\author[J. Van Schaftingen]{Jean Van Schaftingen}
\address[F.~Gmeineder]{Universit\"{a}t Bonn, Mathematisches Institut, Endenicher Allee 60, Bonn, Germany}
\address[B.~Rai\c{t}\u{a}]{University of Warwick, Zeeman Building, Coventry CV4 7HP, United Kingdom}
\address[J.~Van Schaftingen]{Universit\'e catholique de Louvain,
Institut de Recherche en Math\'e\-matique et Physique,
Chemin du Cyclotron 2 bte L7.01.01,
1348 Louvain-la-Neuve,
Belgium}


%\date{\today}

\begin{abstract}
We establish that trace inequalities for vector fields $u\in\hold^\infty_c(\R^n,\R^N)$
\begin{align}\label{eq:abs}\tag{$*$}
	\|D^{k-1}u\|_{\lebe^{\frac{n-s}{n-1}}(\dif\mu)}\leq c\|\mu\|^{\frac{n-1}{n-s}}_{\lebe^{1,n-s}}\|\A[D]u\|_{\lebe^1(\dif\mathscr{L}^n)}
\end{align}
hold if and only if the $k$-th order homogeneous linear differential operator $\A[D]$ on $\R^n$ is elliptic and cancelling, provided that $s<1$, and give partial results for $s=1$, where stronger conditions on $\A[D]$ are necessary. Here, $\|\mu\|_{\lebe^{1,\lambda}}$ denotes the Morrey norm of $\mu$ so that such traces can be taken, for example, with respect to $\mathscr{H}^{n-s}$-measure restricted to fractals of codimension $s<1$. The class of inequalities \eqref{eq:abs} give a systematic  generalisation of \textsc{Adams}' trace inequalities to the limit case $p=1$ and can be used to prove trace embeddings for functions of bounded $\A$-variation,  thereby comprising Sobolev functions and functions of bounded variation or deformation. We also prove a multiplicative version of \eqref{eq:abs}, which implies strict continuity of the associated trace operators on $\bv^\A$. 
\end{abstract}
\maketitle

\setcounter{tocdepth}{1}
%\tableofcontents

%\end{frontmatter}
%\tableofcontents

\section{Introduction}

Traces in function space theory are a weak notion of restriction, which is well-defined and stable under convergence,
and which can be defined on Lebesgue-negligible sets by some weak differentiability or other regularity properties of the functions considered.
In partial differential equations and the calculus of variations, trace theory gives foundation to the prescription of boundary conditions.

The trace problem can be approached through measure theory and harmonic analysis via the cornerstone trace inequality of \textsc{Adams} for Riesz potentials of $\lebe^p$-functions \cite{AdamsTraces1,AdamsTraces2}. To set up the theme, we recall that for $0<\alpha<n$ the $\alpha$-th order Riesz potential $I_{\alpha}$ of a sufficiently integrable measurable map $f\colon\R^{n}\to\R$ is defined by 
\begin{align}\label{eq:Riesz}
I_{\alpha}f(x):=c_{n, \alpha}\int_{\R^{n}}\frac{f(y)}{|x-y|^{n-\alpha}}\dif y,\qquad x\in\R^{n}, 
\end{align}
$c_{n, \alpha } =  \frac{\Gamma((n-\alpha)/2)}{\pi^{n/2}2^\alpha\Gamma(\alpha/2)} > 0$. As proved by \textsc{Adams} \cite{AdamsTraces1,AdamsTraces2}, for $1 < p < \infty$ and $0\leq s<n$ with $0\leq s<\alpha p<n$ and $q=\frac{(n-s)p}{n-\alpha p}$, there exists $c=c(n,p,\alpha,s)>0$ such that 
\begin{align}\label{eq:adams}
	\|I_\alpha f\|_{\lebe^q(\R^{n};\dif\mu)}\leq c\|\mu\|^{1/q}_{\lebe^{1,n-s}(\R^{n})}\|f\|_{\lebe^p(\R^{n})}
\end{align}
holds for all $f\in\lebe^{p}(\R^{n})$ and positive Borel measures $\mu$ on $\R^n$. For $0\leq \lambda \leq n$,  $\|\mu\|_{\lebe^{1,\lambda}(\R^{n})}$ here denotes the Morrey norm of $\mu$, which we recall to be defined as 
\begin{align*}
\|\mu\|_{\lebe^{1,\lambda}(\R^{n})}=\sup_{B}\dfrac{|\mu|(B)}{r(B)^\lambda},
\end{align*}
the supremum ranging over all open balls $B\subset\R^n$; $r(B)$ denotes the radius of the ball $B$. Measures $\mu$ satisfying $\|\mu\|_{\lebe^{1,\lambda}(\R^{n})}<\infty$ are sometimes called $\lambda$-\emph{Ahlfors regular}. Note that, even for $s=0$, the inequality \eqref{eq:adams} does not extend to $p=1$. This can be seen by taking $p=1$, $\mu=\mathscr{L}^n$ and observing that if $f \in \lebe^1 (\R^n)$ and $f > 0$, then $\limsup_{\vert x \vert \to \infty} \vert x\vert^{n - \alpha} (I_\alpha f) (x)>0$
and thus $\|I_\alpha f\|_{\lebe^q(\R^{n};\dif\mu)} = +\infty$ since $q = \frac{n}{n - \alpha}$.

The inequality \eqref{eq:adams} has deep applications in potential theory, cf.~\cite{AH}, and moreover allows to define traces of Sobolev functions $u\in\sobo^{1,p}(\R^{n})$ on suitably regular lower dimensional subsets of $\R^{n}$ provided $1<p<n$. Indeed, given $u\in\hold_{c}^{\infty}(\R^{n})$, in the above setting we may choose $\mu=\mathscr{H}^{n-s}\mres \Sigma$, where $\Sigma \subset \R^n$ is a self-similar fractal set of codimension $s$, as discussed in \cite{Hutchinson}. Such sets are quite well-behaved, in the sense that they are reasonably close to being ``fractional hyperplanes'' by Marstrand's Theorem \cite{Marstrand}. Setting $f:=Du$ in \eqref{eq:adams} and employing the pointwise inequality $|u|\leq c(n)I_{1}(|Du|)$ which follows from the Sobolev integral representations \cite[Sec.~1.1.10,~Thm.~2]{mazya}, the existence of a norm continuous trace operator $\trace_{\Sigma}\colon \sobo^{1,p}(\R^{n})\to \lebe^{q}(\Sigma;\dif\mu)$ follows by a routine approximation argument. With the convention that $\|\cdot\|_{\lebe^p}=\|\cdot\|_{\lebe^p(\R^n;\dif\mathscr{L}^n)}$, we consequently obtain 
for $1<p<n$, $0\leq s<p$, and $q=\frac{(n-s)p}{n-p}$ 
\begin{align}\label{eq:sobolev}
	\|\trace_{\Sigma}(u)\|_{\lebe^q(\dif\mu)}\leq c\|\mu\|^{1/q}_{\lebe^{1,n-s}(\R^{n})}\|D u\|_{\lebe^p(\R^{n},\R^{n})}\quad \text{ for }u\in\sobo^{1,p}(\R^{n}).
\end{align}
Evidently, \eqref{eq:sobolev} generalises the Sobolev embedding $\sobo^{1,p}(\R^{n})\hookrightarrow\lebe^{\frac{np}{n-p}}(\R^{n})$, the latter being retrieved by setting $s=0$, $\Sigma = \R^n$ and $\mu=\mathscr{L}^{n}$ in \eqref{eq:sobolev}. 

\subsection{Limiting $\lebe^{1}$-estimates}

The present work will be concerned with generalisations and aspects of the inequalities \eqref{eq:adams} and \eqref{eq:sobolev} to the limiting case $p=1$. First, due to the lack of strong-type estimates of Riesz potentials on $\lebe^1$, the inequality \eqref{eq:adams} cannot hold for $p=1$, unless the Riesz potential operator is defined on some special strict subspaces of $\lebe^1$. On the other hand, the existence of a trace operator and inequality \eqref{eq:sobolev} can also be proved to hold for $p=1$ as is shown, e.g., in \textsc{Ziemer} \cite[Thm.~5.13.1]{Ziemer} following the foundational work of \textsc{Gagliardo} \cite{Gagliardo}. As we shall elaborate on in more detail below, the validity of \eqref{eq:sobolev} despite the failure of \eqref{eq:adams} is due to the specific structure of the operator $Du\mapsto u$, in turn being a Fourier multiplication operator with symbol homogeneous of degree $(-1)$. 

To systematically approach this theme, let $\A[D]$ be a homogeneous, constant-coefficient differential operators $\A[D]$ of order $k$ on $\R^{n}$ from $V$ to $W$; i.e., $\A[D]$ and its Fourier symbol (characteristic polynomial) have a representation
\begin{align}\label{eq:form}
\A[D]=\sum_{|\alpha|=k}\A_{\alpha}\partial^{\alpha},\qquad\qquad\A[\xi]=\sum_{|\alpha|=k}\xi^{\alpha}\A_{\alpha},\;\;\xi\in\R^n
\end{align}
with linear maps $\A_\alpha\in\mathrm{Lin}(V,W)$ for all multi-indices $\alpha\in\N_0^n$ with $|\alpha|=k$, where $V,\,W$ are finite dimensional normed real vector spaces. Connecting with \eqref{eq:sobolev}, the first main objective of this work is to study inequalities of the form 
\begin{align}\label{eq:main_ineq}
	\|D^{k-1}u\|_{\lebe^q(\dif\mu)}\leq c\|\mu\|^{1/q}_{\lebe^{1,n-s}}\|\A[D] u\|_{\lebe^1(\dif\mathscr L^n)}\quad \text{ for }u\in\hold^\infty_c(\R^n,V),
\end{align}
where $0\leq s<1$ and $q=\frac{n-s}{n-1}$. Note that, by the celebrated {Ornstein} Non-inequality \cite{Ornstein,KK,KSW,KW,CFM}, there is \emph{no} constant $c>0$ such that $\|D^{k}u\|_{\lebe^{1}}\leq c\|\A[D]u\|_{\lebe^{1}}$, except in trivial cases. Hence even for $k=1$, \eqref{eq:main_ineq} is not a consequence of \eqref{eq:sobolev}.

The framework of \eqref{eq:form}  is particularly motivated by earlier examples of Sobolev-type inequalities for Hodge systems \cite{BB3,LaSt,VSforms} characterised in \cite{VS}, the study of lower semi-continuity for variational problems of linear growth \cite{BDG}, and applications in plasticity, fracture mechanics and image reconstruction \cite{AnGi,StTe,ChCr,DaFoLi}. To contexualise our objective, we first recall from \cite[Thm.~1]{Hoermander}, \cite[Def.~1.7.1]{Spencer}, \cite[Def.~26]{DGK}, \cite[Def.~1.1]{VS} that $\A[D]$ is called (overdetermined) \emph{elliptic} provided the symbol map $\A[\xi]\colon V\to W$ is injective for all $\xi\in\R^{n}\setminus\{0\}$. The ellipticity assumption is a standard condition in the context of linear coerciveness estimates. As shown by the third author in \cite[Thm.~1.3]{VS}, the following generalisation of the Sobolev-Gagliardo-Nirenberg inequality
\begin{align}\label{eq:VS}
\|D^{k-1}u\|_{\lebe^{\frac{n}{n-1}}}\leq c\|\A[D]u\|_{\lebe^1}\quad\text{ for }u\in\hold^\infty_c(\R^n,V)
\end{align}
holds if and only if $\A[D]$ is elliptic \emph{and cancelling}. Here, we say that the operator $\A[D]$ given by \eqref{eq:form} is \emph{cancelling} if and only if 
\begin{align}\tag{C}\label{eq:cancelling}
\bigcap_{\xi\in\R^{n}\setminus\{0\}}\mathrm{im\,}\A[\xi]=\{0\}.
\end{align}
The class of elliptic and cancelling operators is a strict subclass of the elliptic operators, as can be seen by the derivative on \(\R\), at the Wirtinger derivatives \(\frac{1}{2}(\partial_1 + i \partial_2)\) on $\R^2 \simeq \C$ or the Laplacian \(-\Delta\) on \(\R^n\). As the first main result of this work, we establish that ellipticity and cancellation are necessary and sufficient for \eqref{eq:main_ineq} to hold:
\begin{theorem}\label{thm:int}
	Let $n\geq 2$, $k\geq 1$ and $\A[D]$ be as in \eqref{eq:form}. Moreover, let $0\leq s<1$ and $q:=\frac{n-s}{n-1}$. Then the following are equivalent: 
	\begin{enumerate}
		\item $\A[D]$ is elliptic and cancelling.
		\item There exists a constant $c>0$ such that for all $u\in\hold_{c}^{\infty}(\R^{n},V)$ and all positive Borel measures $\mu$ on $\R^n$ there holds 
		\begin{align*}
		\|D^{k-1}u\|_{\lebe^q(\dif\mu)}\leq c\|\mu\|^{1/q}_{\lebe^{1,n-s}}\|\A[D]u\|_{\lebe^1(\dif\mathscr{L}^n)}.
		\end{align*}
	\end{enumerate}
\end{theorem}
As for comparison with \eqref{eq:adams}, it is not difficult to adapt our method to obtain that, if $0\leq s< \alpha<n$, $q=\frac{n-s}{n-\alpha}$, and $\A[D]$ is elliptic, we have that $\A[D]$ is cancelling if and only if
\begin{align}\label{eq:adamsformula1}
	\|I_\alpha \A[D]u\|_{\lebe^q(\dif\mu)}\leq c\|\mu\|^{1/q}_{\lebe^{1,n-s}}\|\A[D]u\|_{\lebe^1}, \quad\text{ for }u\in\hold^\infty_c(\R^n,V).
\end{align}
The inequality \eqref{eq:adamsformula1} seems to be the first systematic generalisation of \textsc{Adams}' trace inequality \eqref{eq:adams} to the case $p=1$. Let us note that subject to ellipticity and cancellation, suitable estimates on lower order derivatives can equally be obtained, see Proposition~\ref{prop:suff_EC}.

It is important to mention that by {Ornstein}'s Non-inequality, the inequality of Theorem~\ref{thm:int} is a strict improvement of \eqref{eq:sobolev}. For comparison, in the case where $1<p<n$, the corresponding claim of Theorem~\ref{thm:int} is rather straightforward as ellipticity of $\A[D]$ suffices to reduce the analogue of \eqref{eq:main_ineq} to \eqref{eq:sobolev} by the {H\"{o}rmander-Mihlin} multiplier theorem; see Lemma~\ref{lem:WAptracedefine} for the quick argument. Also note that in view of Theorem~\ref{thm:main2} below, Theorem~\ref{thm:int} is optimal in the sense that \emph{it does not extend to $s=1$}, as we will discuss later. Despite having singled out Theorem~\ref{thm:int} for future reference, it can be sharpened and is implied by the following multiplicative trace inequality, which is in the spirit of \cite[Sec.~1.4.7]{mazya}:
\begin{theorem}\label{thm:main1}
In the situation of Theorem~\ref{thm:int}, the following are equivalent: 
\begin{enumerate}
	\item\label{itm:thm1.1_b} $\A[D]$ is elliptic and cancelling.
	\item\label{itm:thm1.1_a} Let $s\frac{n-1}{n-s}<\theta\leq 1$. Then there exists a constant $c>0$ such that for all $u\in\hold_{c}^{\infty}(\R^{n},V)$ and all positive Borel measures $\mu$ on $\R^n$ there holds 
\begin{align}\label{eq:multiplicativetraceinequality}
\|D^{k-1}u\|_{\lebe^q(\dif\mu)}\leq c\|\mu\|^{1/q}_{\lebe^{1,n-s}}\|D^{k-1}u\|^{1-\theta}_{\lebe^{\frac{n}{n-1}}(\dif\mathscr{L}^n)}\|\A[D]u\|^\theta_{\lebe^1(\dif\mathscr{L}^n)}.
\end{align}
\end{enumerate}
\end{theorem}
The theorem is equally optimal in the sense that \emph{no such multiplicative inequality can hold for} $s=1$. Indeed, as one of the main points of the paper, we shall establish that the specific multiplicative form of \eqref{eq:multiplicativetraceinequality} directly translates to the so-called $\A$-strict continuity of the trace operator associated with $\mu$. On the other hand, for $s=1$ the trace operator \emph{is never  $\A$-strictly continuous}, cf. Remark~\ref{rem:nonstrictlycontinuous} and Section~\ref{sec:BVAspaces} for the requisite terminology. In particular, we see that the fact that the range of $\theta$ in \ref{itm:thm1.1_a} is empty for $s=1$ is phenomenological.

At the endpoint $s=1$, Theorem~\ref{thm:int} cannot be generalised by easy means\footnote{In fact, to the best of our knowledge, Theorem~\ref{thm:int} is only known in the limiting case $s=1$ for $\A[D]=D^k$, which follows from the work of Meyers and Ziemer \cite{MeZi}. There, the coarea formula is crucially used and there seems to be no simple replacement of this tool for other operators $\A[D]$.}. Below we show that the class of operators admitting the suitable endpoint estimate is, in general, \emph{strictly} smaller than the class of elliptic and cancelling operators%; a full characterisation is available for first order operators
. As a metaprinciple, lower codimensions $0<s<1$ require weaker conditions on $\A[D]$ for the respective trace inequalities to hold whereas the \emph{borderline case} $s=1$ necessarily requires stronger conditions on $\A[D]$. 
\begin{proposition}\label{prop:int_tr_implies_En-2canc}
Let $\A[D]$ be as in \eqref{eq:form}. Suppose that for every $(n-1)$-dimensional hyperplane $\Sigma\subset\R^{n}$ there exists a constant $c>0$ such that for all $u\in\hold_{c}^{\infty}(\R^{n},V)$ there holds 
\begin{align}\label{eq:int_tr_k}
\|D^{k-1}u\|_{\lebe^1(\Sigma;\dif\mathscr{H}^{n-1})} \leq c \|\A[D]u\|_{\lebe^1(\R^n;\dif\mathscr{L}^{n})}. 
\end{align}
Then $\A[D]$ is elliptic and satisfies the \emph{strong cancellation} condition
\begin{align}\tag{SC}\label{eq:n-2_canc_intro}
\bigcap_{\xi\in H\setminus\{0\}}\mathrm{im\,}\A[\xi]=\{0\}\text{ for any subspace $H\leq \R^n$ with $\dim H=2$.}
\end{align}
\end{proposition}
Obviously, condition~\eqref{eq:n-2_canc_intro} reduces to cancellation if $n=2$. For $n>2$, to see that condition~\eqref{eq:n-2_canc_intro} is strictly stronger than cancellation in the class of elliptic operators, one considers an elliptic operator $\mathbb{B}_1[D]$ on $\R^2$ from $\R^2$ to $\R^2$ (e.g., $\mathbb{B}_1[D]=\Delta^{k/2}$ if $k$ is even and $\mathbb{B}_1[D]=\Delta^{(k-1)/2}\circ(\di,\,\mathrm{curl})$ if $k$ is odd) and sets
\begin{align*}
\A[D]=\left(
\begin{matrix}
\mathbb{B}_1[D]&0\\
0&\mathbb{B}_2[D]
\end{matrix}
\right),
\end{align*}
where $\mathbb{B}_2[D]$ is a $k$-th order elliptic and cancelling operator on $\R^{n-2}$ (e.g., the $k$-th gradient). We shall give an explicit example in Example~\ref{ex:ESCnotCell}. Condition~\eqref{eq:n-2_canc_intro} first appeared in the context of differential operators in the second author's technical report \cite[Sec.~5.2,~Prop.~2.6]{Raita18}. Similar ideas appeared much earlier in \cite[Thm.~3]{RW} in connection with dimensional estimates for measures satisfying certain restrictions in Fourier space (see also the recent \cite[Def.~1.4]{AW}, which is related to our context via the substitution $\phi(\xi)=\mathrm{im\,}\A[\xi]$). As for a possible converse of Proposition~\ref{prop:int_tr_implies_En-2canc}, we demonstrate that this is true indeed provided that the operator $\A[D]$ is \emph{of first order}:
\begin{theorem}\label{thm:main2}
Let $\A[D]$ be a \emph{first order} operator as in \eqref{eq:form}. Then the following are equivalent: 
\begin{enumerate}
\item\label{it:int_k=1_est} For every $(n-1)$-dimensional hyperplane $\Sigma\subset\R^{n}$ there exists a constant $c>0$ such that for all $u\in\hold_{c}^{\infty}(\R^{n},V)$ there holds 
\begin{align*}
\|u\|_{\lebe^1(\Sigma;\dif\mathscr{H}^{n-1})} \leq c \|\A[D]u\|_{\lebe^1(\R^n;\dif\mathscr{L}^{n})}. 
\end{align*}
\item\label{it:int_k=1_En-2C} $\A[D]$ is elliptic and \eqref{eq:n-2_canc_intro} holds.
\item\label{it:int_k=1_est_strong} For every $(n-1)$-dimensional hyperplane $\Sigma\subset\R^{n}$ there exists a constant $c>0$ such that for all $u\in\hold_{c}^{\infty}(\R^{n},V)$ there holds 
\begin{align*}
\|u\|_{\lebe^1(\Sigma;\dif\mathscr{H}^{n-1})} \leq c \|\A[D]u\|_{\lebe^1(\Sigma^+;\dif\mathscr{L}^{n})}. 
\end{align*}
Here $\Sigma^+$ is a half-space with boundary $\Sigma$.
\item\label{it:int_k=1_Cell} $\A[D]$ is \emph{$\mathbb{C}$-elliptic}, i.e., $\ker_{\mathbb{C}}\A[\xi]=\{0\}$ for all $\xi\in\mathbb{C}^n\zeroset$.
\end{enumerate}
\end{theorem}
In proving the previous theorem, we extend and augment ideas given by \textsc{Breit, Diening} and the first author in \cite{BDG}. Modifying the approach given in \cite{BDG}, it is further possible to establish \emph{for $k$-th order $\mathbb{C}$-elliptic operators} that the \emph{interior trace inequality} \eqref{eq:int_tr_k} remains valid. This is a consequence of an \emph{exterior trace inequality}, cf.~ Theorem~\ref{thm:ext} below. Hinging on the linearity of $\xi\mapsto\A[\xi]$, the previous theorem does not immediately generalise to $k$-th order operators, see Open Problem~\ref{openproblem} below. %Strikingly different from the Sobolev space case, exterior and interior traces are not equally easily treatable; see Section~\ref{sec:ext} for a general discussion. 
\subsection{Trace embeddings for function spaces}
As already alluded to above, Theorems~\ref{thm:int}-\ref{thm:main2} have counterparts in the trace theory for function spaces. Let $\A[D]$ be of the form \eqref{eq:form}. Introduced in \cite{BDG} and studied in \cite{ChCr,DaFoLi,GR,GR_diff,Raita18,R}, we recall that the space of \emph{functions of bounded $\A$-variation} is given by
\begin{align*}
	\bv^\A(\R^n):=\{u\in\sobo^{k-1,1}(\R^n,V)\colon \A[D]u\in\mathscr{M}(\R^n,W)\}.  
\end{align*}
These spaces provide a unifying treatment of well-known spaces such as $\bv(\R^{n})$ or $\bd(\R^{n})$. To connect with the theme from above, obtaining a definition of trace operator from Theorem~\ref{thm:int} in itself is not obvious and cannot be accessed by the coarea-formula approach of \textsc{Ziemer} \cite[Thm.~5.2.13]{Ziemer}. Similar to the $\bv$-case, the norm topology on $\bv^{\A}$ is not well-suited for most applications; in particular, note that elements of $\bv^{\A}(\R^{n})$ cannot be smoothly approximated in this topology. The correct substitute topology is that induced by the $\A$-strict metric $\dif_{\A}(u,v):=\|u-v\|_{\sobo^{k-1,1}(\R^{n},V)}+||\A[D]u|(\R^{n})-|\A[D]v|(\R^{n})|$. This topology not only ensures continuity of suitably convex functionals on $\bv^{\A}(\R^{n})$ (cf.~\cite{Resh,BDG}), but also admits smooth approximation of $\bv^{\A}$-maps. In this respect, the concluding main result of the present paper is the following theorem, for ease of exposition stated for first order operators.
\begin{theorem}[Traces and $\A$-strict continuity]\label{thm:strict}
In the setting of Theorem~\ref{thm:int}, let $\A[D]$ be a first order elliptic and cancelling operator and $\mu\in\lebe^{1,n-s}(\R^{n})$. Then the following hold: 
\begin{enumerate}
\item\label{itm:tr1intro} There exists a norm continuous linear trace embedding operator 
\begin{align*}
\trace_{\mu}\colon\bv^{\A}(\R^{n})\hookrightarrow \lebe^{q}(\R^{n};\dif\mu).
\end{align*}
 \item\label{itm:tr2intro} Moreover, the operator $\trace_{\mu}$ from \ref{itm:tr1intro} is $\A$-strictly continuous. 
\end{enumerate}
\end{theorem}
Suitable higher order variants can be easily formulated. Whereas \ref{itm:tr1intro} is known for $\bv(\R^{n})$, it seems to be new even for $\bd(\R^{n})$. To the best of our knowledge, \ref{itm:tr2intro} seems to be a novel result even for $\bv(\R^{n})$. The strengthening of continuity properties from \ref{itm:tr1intro} to \ref{itm:tr2intro} can be seen in parallel with the strengthening of Theorem~\ref{thm:int} by Theorem~\ref{thm:main1}. Note that the foremost issue in proving \ref{itm:tr2intro} is that the $\A$-strict convergence is nonlinear\footnote{In the sense that $u_{j}\to u$ and $v_{j}\to v$ $\A$-strictly do \emph{not} imply that $u_{j}+v_{j}\to u+v$ $\A$-strictly.}. If $s=1$, then it is possible to give a variant of \ref{itm:tr1intro} for $\mathbb{C}$-elliptic operators $\A[D]$, but even here $\A$-strict continuity is never achievable.
\subsection{Structure of the paper}
This paper is organised as follows: In Section~\ref{sec:prelims} we fix notation and gather preliminary results on function spaces and potential theory. Section~\ref{sec:notions} gathers and connects algebraically the notions for differential operators to be used throughout. In Section~\ref{sec:EC}, we prove Theorems~\ref{thm:int} and \ref{thm:main1}, while Section~\ref{sec:s=1} is devoted to Theorem~\ref{thm:main2}. The aforementioned trace theory for $\bv^{\A}$-spaces then is given in Section~\ref{sec:selected}, and the appendix gathers background results on vectorial measures.
 
{\small
\subsection*{Acknowledgment} The authors thank Panu Lahti and Ben Hambly for helpful discussions on fractal sets and measures. B.R. acknowledges the hospitality of HCM Bonn, where a significant part of this work was completed. J.V.S. acknowledges the hospitality of the Mathematical Institute of the University of Oxford where part of this work was completed. This project has received funding from the European Research Council (ERC) under the European Union's Horizon 2020 research and innovation programme under grant agreement No 757254 (SINGULARITY). This
work was supported by Engineering and Physical Sciences Research Council Award
EP/L015811/1.}

\section{Preliminaries}\label{sec:prelims}
\subsection{General notation}\label{sec:notation}
Given $x_{0}\in\R^{n}$ and $r>0$, we denote $\ball(x_{0},r):=\{x\in\R^{n}\colon\;|x-x_{0}|<r\}$ the open ball with radius $r>0$ centered at $x_{0}$. We will work with \emph{double cones}, by which we mean sets $\mathcal{C}=\{x_0+tx\colon x\in\mathcal{S},t\in\R\}$, where $x_0\in\R^n$ is the \emph{apex} of $\mathcal{C}$ and $\mathcal{S}\subset\mathbb{S}^{n-1}$ denotes a non-empty, relatively open \emph{spherical cap}, i.e., the non-empty intersection of $\mathbb{S}^{n-1}$ with an open ball in $\R^n$.  %A positive cone $K$ in $\R^{n}$ is a set such that, if $x,y\in K$ and $\alpha,\beta\in\R_{\geq 0}$, then $\alpha x+\beta y\in K$, and $K\cap (-K)=\{0\}$. 
For an open set $\Omega\subset\R^{n}$ and a finite dimensional real vector space $E$, we denote $\mathscr{M}(\Omega,E)$ the finite $E$-valued Radon measures on $\Omega$; for more background information on vectorial measures, see, e.g., \cite[Chpt.~1]{AFP}. For $\mu\in\mathscr{M}(\Omega,E)$ and $A\in\mathscr{B}(\Omega)$ (with the Borel $\sigma$-Algebra $\mathscr{B}(\Omega)$ on $\Omega$), we denote $\mu\mres A:=\mu(A\cap -)$ the restriction of $\mu$ to $A$. The $n$-dimensional Lebesgue and $\alpha$-dimensional Hausdorff measures, $0\leq \alpha \leq n$, are denoted $\mathscr{L}^{n}$ and $\mathscr{H}^{\alpha}$, respectively. Given a $\mathscr{L}^{n}$-measurable map $u\colon\R^{n}\to V$, we recall that its \emph{precise representative} is defined by 
\begin{align*}
u^{*}(x):=\begin{cases}\displaystyle \lim_{R\searrow 0}\dashint_{\ball(x,R)}u\dif y&\;\text{provided this limit exists and is finite}, \\ 0 &\;\text{otherwise}. \end{cases} 
\end{align*}
Denoting as usual $S_{u}$ the Lebesgue discontinuity points of $u$, the right-hand side of the previous definition exists for all $x\in S_{u}^{c}$. For a given map $u\in\lebe^{1}(\R^{n},E)$, we shall work with the Fourier transform $\mathscr{F}u$ defined by 
\begin{align*}
\mathscr{F}u(\xi):=\frac{1}{(2\pi)^{\frac{n}{2}}}\int_{\R^{n}}u(x)e^{-\imag x \cdot \xi}\dif x,\qquad \xi\in\R^{n}.
\end{align*}
In this connection, we will write $\mathscr{S}(\R^n,E)$ for the Schwarz class of rapidly decreasing functions and $\mathscr{S}^\prime(\R^n,E)$ for its linear topological dual, the space of tempered distributions. Further, we denote the set of $E$-valued polynomials on $\R^{n}$ of degree at most $d$ by $\mathscr{P}_{d}(\R^{n},E)$, and set $\mathscr{P}(\R^{n},E):=\bigcup_{d}\mathscr{P}_{d}(\R^{n},E)$. We use the notation ``$\leq$'' for the linear subspace inclusion relation, whenever it does not denote inequality between numbers. Throughout, $c>0$ denotes a constant that does not depend on any quantity that may change from line to line. Also, $c_n$ denotes a constant that depends on the space dimension $n$ only. Finally, we clarify that $x\cdot\xi=x\cdot\Re\xi+\imag x\cdot \Im\xi$ whenever $x\in\R^n$ and $\xi\in\C^n$, where the dot product of real vectors is defined in a standard way.
\subsection{Function spaces and potential theory}
In this section we record various background results on $\lebe^{p}$, Riesz potential, and Triebel-Lizorkin   spaces that shall be required in the sequel; for more detail, the reader is referred to \textsc{Triebel} \cite{Triebel} and \textsc{Adams \& Hedberg} \cite{AH}.  We begin with the following lemma due to \textsc{Brezis \& Lieb}:
\begin{lemma}[{\cite{BrezisLieb}}]\label{lem:BrezisLieb}
Let $\Omega\subset\R^{n}$ be open and suppose that $1\leq p < \infty$. If $u,u_{1},u_{2},...\in\lebe^{p}(\Omega,V)$ satisfy (i) $u_{j}\rightharpoonup u$ in $\lebe^{p}(\Omega,V)$ and (ii) $u_{j}\to u$ pointwisely $\mathscr{L}^{n}$-a.e., then there holds 
\begin{align}
 \lim_{j\to\infty}\bigl(\|u_{j}\|_{\lebe^{p}(\Omega,V)}^{p}-\|u_{j}-u\|_{\lebe^{p}(\Omega,V)}^{p}\bigr)=\|u\|_{\lebe^{p}(\Omega,V)}^{p}.
\end{align}
\end{lemma}
We now briefly recall the definition of homogeneous Triebel-Lizorkin spaces as they will play an important auxiliary role later on. To this end, let $\Phi\in\mathscr{S}(\R^{n})$ be such that $\spt(\widehat{\Phi})\subset\ball(0,1)=:\ball_{0}$ such that $\widehat{\Phi}=1$ on $\ball(0,\frac{1}{2})=:\ball_{1}$. A dyadic resolution of unity then is obtained by virtue of $\varphi_{j}:=2^{jn}\Phi(2^{j}x)-2^{(j-1)n}\Phi(2^{j-1}x)$ for $j\in\mathbb{Z}$. Given $s\in\R$ and $1<p,q<\infty$, the \emph{homogeneous Triebel-Lizorkin space} $\dot{\trli}{_{p,q}^{s}}(\R^{n},V)$ is defined as the linear space of all $u\in(\mathscr{S}'/\mathscr{P})(\R^{n},V)$ such that 
\begin{align}\label{eq:TriebelLizorkinnorm}
\|f\|_{\dot{\trli}{_{p,q}^{s}}(\R^{n},V)}:=\Big\lVert\Big(\sum_{j\in\mathbb{Z}}2^{sjq}|(\mathscr{F}^{-1}(\varphi_{j}\mathscr{F}f))(\cdot)|^{q} \Big)^{\frac{1}{q}}\Big\rVert_{\lebe^{p}(\R^{n})}<\infty.
\end{align}
The inhomogeneous variant $\trli_{p,q}^{s}(\R^{n},V)$ is obtained by replacing $\mathscr{S}'/\mathscr{P}$ by $\mathscr{S}'$ and, in \eqref{eq:TriebelLizorkinnorm}, $j\in\mathbb{Z}$ by $j\in\mathbb{N}_{0}$. Following \cite[Ch.~5.2.3]{Triebel}, we define the Riesz potential operators $I_\alpha$ for $\alpha\in\R$ by
\begin{align*}
	\widehat{I_\alpha f}(\xi):=|\xi|^{-\alpha}\widehat{f}(\xi),\quad\xi\in\R^n
\end{align*}
for Schwartz maps $f\in\mathscr{S}(\R^{n},V)$. In particular, if $0<\alpha<n$, we retrieve the standard representation of the Riesz potential operators given by \eqref{eq:Riesz}. With this definition, the Riesz potential operators satisfy the semigroup property $I_\alpha I_\beta=I_{\alpha+\beta}$ for $\alpha,\beta>0$ with $\alpha+\beta<n$. The instrumental continuity property of Riesz potentials then is given in:
\begin{lemma}[{\cite[Thm.~5.2.3]{Triebel}}]\label{lem:triebel_lizorkin}
	Let $-\infty<s<\infty$, $1<p<\infty$. Then we have the equivalence of norms
	\begin{align*}
		\|f\|_{\dot{\trli}{_{p,2}^{s}}(\R^n)}\sim \|I_{-s}f\|_{\lebe^p(\R^n)}.
	\end{align*}
	Moreover, if $s=m$ is a natural number, we have that the quantity above is equivalent to the $\dot{\sobo}{^{m,p}}$ (semi-)norm, with the convention that $\dot{\sobo}{^{0,p}}=\lebe^p$. 
\end{lemma}
We will also use an interpolation result for homogeneous Triebel-Lizorkin spaces:
\begin{lemma}[{\cite[Ch.~5.2.5,~Thm.~2.4.7]{Triebel}}]\label{lem:interpolation}
Let $-\infty<s_1,s_2<\infty$, $1<p_1,p_2<\infty$, and $0\leq \theta\leq 1$. Then
\begin{align*}
	\|f\|_{\dot{\trli}{_{p,2}^{s}}(\R^n)}\leq c\|f\|^{1-\theta}_{\dot{\trli}{_{p_1,2}^{s_1}}(\R^n)}\|f\|^{\theta}_{\dot{\trli}{_{p_2,2}^{s_2}}(\R^n)}\quad\text{ for }f\in\hold^\infty_c(\R^n),
\end{align*}
where
\begin{align*}
	s=(1-\theta)s_1+\theta s_2,\quad\frac{1}{p}=\frac{1-\theta}{p_1}+\frac{\theta}{p_2}.
\end{align*}
\end{lemma}
We conclude this subsection with a result on Lebesgue points in a form following directly from \cite[Thm.~6.2.1]{AH}: 
\begin{lemma}\label{lem:Triebeltrace}
Let $0<s<1$, $0<t<1$ and $1<p<\infty$ be such that $t<sp$. If $\mu\in\lebe^{1,n-t}(\R^{n})$, then there holds $\mu(S_{u})=0$ for all $u\in \trli_{p,p}^{s}(\R^{n},V)$. 
\end{lemma}
\section{Notions for Differential Operators}\label{sec:notions}
In this section, we compactly gather and introduce the necessary notions for $k$-th order differential operators $\A[D]$ of the form \eqref{eq:form}. These are stated in terms of the symbol map $\R^{n}\ni\xi\mapsto\A[\xi]$, where $n\geq 2$ throughout, and we say that $\A[D]$ is 
\begin{enumerate}
	\item\label{itm:elliptic} \emph{elliptic} whenever $\ker_{\R}\A[\xi]=\{0\}$  for any $\xi\in\R^n\setminus\{0\}$.
	\item\label{itm:Celliptic} \emph{$\C$-elliptic} whenever $\ker_{\C}\A[\xi]=\{0\}$  for any $\xi\in\C^n\setminus\{0\}$.
	\item\label{itm:cancelling} \emph{cancelling} whenever $\bigcap_{\xi\in\R^n\setminus\{0\}}\mathrm{im\,}\A[\xi]=\{0\}$.
	\item\label{itm:stronglycancelling} \emph{strongly cancelling} whenever $\bigcap_{\xi\in H\setminus\{0\}}\mathrm{im\,}\A[\xi]=\{0\}$ for any subspace $H\leq \R^n$ with $\dim H=2$.
\end{enumerate}
The difference between the conditions in \ref{itm:cancelling} and \ref{itm:stronglycancelling} can be conveniently represented as 
\begin{align*}
&\text{(cancellation)}\;\;\;\;\;\;\;\;\;\;\;\Leftrightarrow\;\;\;\bigcup_{H = \R^{n}}\bigcap_{\xi\in H\setminus\{0\}}\image \A[\xi]=\{0\},\\ &\text{(strong cancellation)}\Leftrightarrow\bigcup_{\substack{H\leq\R^{n} \\ \dim(H)=2}}\bigcap_{\xi\in H\setminus\{0\}}\image\A[\xi]=\{0\}.
\end{align*}
It is clear that $\mathbb{C}$-ellipticity implies ellipticity, and that if $\A[D]$ is elliptic and strongly cancelling, then it is elliptic and cancelling. The fact that $\C$-ellipticity implies cancellation was first observed in \cite[Lem.~3.2]{GR}, where it was established by use of an analytic characterisation of  cancellation \cite[Prop.~6.1]{VS}. The fact that $\C$-ellipticity implies strong cancellation was first noticed by the second author in \cite[Prop.~2.6]{Raita18}, where a slightly different characterisation of the cancellation is employed \cite[Lem.~2.5]{R}. Here we give the first purely algebraic proof of this fact:
\begin{proposition}\label{lem:FDN_implies_EC}
Let $k\geq 1$ and suppose that $\A[D]$ is a $k$-th order, $\mathbb{C}$-elliptic differential operator of the form \eqref{eq:form}. Then $\A[D]$ is strongly cancelling.
\end{proposition}
For the proof of Proposition~\ref{lem:FDN_implies_EC} and as a result on its own right, we require the following characterisation of $\mathbb{C}$-elliptic operators. This appears as an extension of the arguments given by  \textsc{Smith} \cite{Smith}:
\begin{proposition}\label{prop:smith}
	Let $\A[D]$ be as in \eqref{eq:form}. The following are equivalent:
	\begin{enumerate}
		\item\label{itm:smith_a} $\A[D]$ is $\C$-elliptic.
		\item\label{itm:smith_b} There exists an integer $d$ and a homogeneous linear differential operator $\mathbb{B}[D]$ on $\R^n$ from $W$ to $V\odot^{d-k}\R^n$ such that $D^d=\mathbb{B}[D]\circ\A[D]$.
		\item\label{itm:smith_c} $\A[D]$ has \emph{finite dimensional nullspace}, i.e., $\dim\{u\in\mathscr{D}^\prime(\R^n,V)\colon \A[D]u=0\}<\infty$.
	\end{enumerate}
\end{proposition}
Here $V\odot^l\R^n$ denotes the space of $V$-valued, symmetric, $l$-linear maps on $\R^n$.
\begin{figure}
\centerline{
  % Resize it to 5cm wide.
  \resizebox{9.5cm}{!}{
    \begin{tikzpicture}[
      scale=1,
      level/.style={thick},
      virtual/.style={thick,densely dashed},
      trans/.style={thin,double,<-,shorten >=2pt,shorten <=2pt,>=stealth},
      classical/.style={thin,double,<->,shorten >=4pt,shorten <=4pt,>=stealth}
    ]
    \draw[virtual] (2cm,8em) -- (4cm,8em) node[midway,above] {$\mathbb{C}$-ellipticity};
    \draw[virtual] (3.5cm,-2em) -- (4.5cm,-2em) node[midway,below] {\text{Ellipticity and Cancellation}};
    \draw[virtual] (1cm,1em) -- (0cm,1em) node[midway,above] {\text{Ellipticity and Strong Cancellation}};
    % Draw the transitions.
    \draw[trans,red] (2.25cm,8em) -- (0.5cm,2.5em)  node[midway,left] {{$k=1$}};
    \draw[trans,black] (0.7cm,2.5em) -- (2.45cm,8em)  node[midway,left] {};
    \draw[trans,red] (3.5cm,8em) -- (4cm,-2em) node[midway,right] {{ $n=2$}};
    \draw[trans] (3.75cm,-1.75em) -- (0.5cm,0.75em) node[midway,below] {};
\draw[->] (-3.5cm,-7em) -- (-3.5cm,12em);
\draw[virtual] (3cm,-7em) -- (5cm,-7em) node[midway,above] {\text{Ellipticity}};
 \draw[trans] (4cm,-5.5em) -- (4cm,-3.5em)  node[midway,left] {};
 \node[rotate=90] at (-3cm,2em) {strength of the conditions};
    \end{tikzpicture}
  }
}
\caption{Connection of the single notions for differential operators of the form \eqref{eq:form}, contextualising Proposition~\ref{lem:FDN_implies_EC}, Lemma~\ref{lem:smahoz} and Example~\ref{ex:ESCnotCell} with previous results, cf. \cite{GR}; black arrows hold unconditionally for $k\geq 1$ and $n\geq 2$, whereas red arrows only hold under the conditions as specified beneath.}
\end{figure}
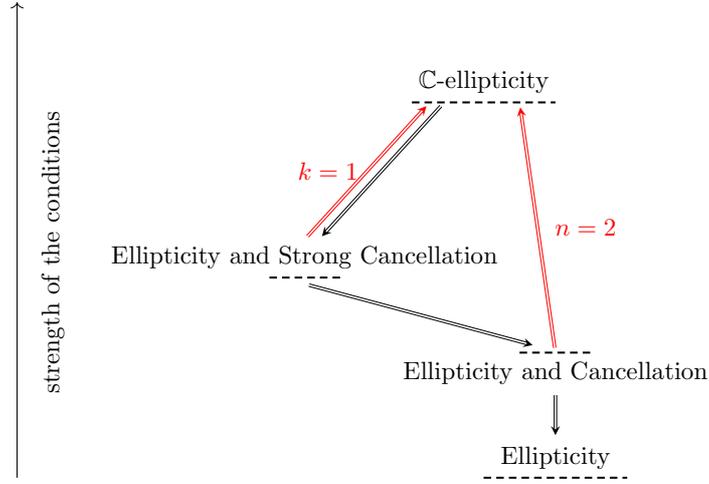
\begin{proof}
Suppose that $\A[D]$ is $\C$-elliptic. To prove that \ref{itm:smith_b} holds, we follow the ideas in \cite[Thm.~4.1]{Smith}, where the argument is formulated for convolution kernels instead of polynomials. We write $N=\dim V$ and $\dim W=m$ and represent $\A[\xi]=(a_{ji}(\xi))_{j=1,\ldots m,i=1,\ldots N}$, which we view as matrix valued complex polynomial. We write $\I=(j_1,\ldots,j_N)$ for a typical multi-index of length $N$ and entries $1\leq j_l\leq m$, $l=1,\ldots N$. These multi-indices correspond to all $N\times N$ sub-matrices of $\A[\xi]$, i.e., $\mathcal{D}_{\I}(\xi)=(a_{j_l,i}(\xi))_{l,i=1,\ldots,N}$. We write $d_\I(\xi)=\det\mathcal{D}_{\I}(\xi)$ and $d_\I^{li}(\xi)$ for the $(l,i)$-minor of the matrix $\mathcal{D}_\I(\xi)$.

With this setup, the $\C$-ellipticity assumption translates into the fact that the homogeneous scalar polynomials $d_{\I}$ have no non-trivial common complex zeroes. Since the polynomials $\xi_g$, $g=1\ldots n$, vanish on the set of common zeroes of $\{d_\I\}_{\I}$, we have by Hilbert's Nullstellensatz \cite[Thm.~4.1.2]{CLO} that there exists a power $p_g$ such that $\xi_g^{p_g}$ is in the ideal generated by $d_\I$. By taking $d=\max_{g=1,\ldots,n} p_g$, we can write
\begin{align}\label{eq:smith1}
	\xi^\alpha v_i=\sum_{\I}B^{\alpha}_\I(\xi)d_\I(\xi)v_i=\sum_{\I,h}B^{\alpha}_\I(\xi)d_\I(\xi)\delta_{ih}v_h,
\end{align}
where $v\in V$, $i=1\ldots N$, $|\alpha|=d$, $B_\I^\alpha(\xi)$ are complex polynomials as given by the Nullstellensatz, and $\delta$ (only here) denotes the Kronecker delta. By the adjugate (or cofactor) formula applied to the matrix $\mathcal{D}_{\I}(\xi)$, we have that 
\begin{align}\label{eq:smith2}
d_\I(\xi)\delta_{ih}=\sum_{l}a_{j_lh}(\xi)d_\I^{li}(\xi),
\end{align}
where $\I=(j_l)_{i=1}^N$. Substituting \eqref{eq:smith2} in \eqref{eq:smith1}, we get that 
\begin{align*}
	\xi^\alpha v_i&=\sum_{\I,h}B^{\alpha}_\I(\xi)\sum_{l}a_{j_lh}(\xi)d_\I^{li}(\xi)v_h=\sum_{\I,l}B^{\alpha}_\I(\xi)d_\I^{li}(\xi)\sum_{h}a_{j_lh}(\xi)v_h\\
	&=\left(\sum_{\I}B^{\alpha}_\I(\xi)P^i_\I(\xi)\right)\A[\xi]v,
\end{align*}
where $P^i_\I(\xi)\in\lin (W,\R)$ is the linear map that deletes the $m-N$ entries indexed by $\I^c$ of a vector $w\in W$ and multiplies the entries $j_l\in \I$, $l=1,\ldots, N$ by $d^{li}_\I(\xi)$. The proof of \ref{itm:smith_b} is complete.

Next, if \ref{itm:smith_b} holds, it is clear that $\A[D]u=0$ implies $D^d u=0$, so $u$ is a polynomial of degree at most $d-1$, hence \ref{itm:smith_c} easily follows. Finally, if \ref{itm:smith_c} holds, we assume for contradiction that \ref{itm:smith_a} fails, so that there exist non-zero $x\in\C^n$ and $v\in V+\imag V$ such that $\A[\xi]v$. Then one shows that the complex plane waves $u(x)=f(x\cdot\xi)v$ for $x\in\R^n$ and \emph{holomorphic} $f\colon\C\rightarrow\C$ are distributional solutions of $\A[D]u=0$, which implies failure of \ref{itm:smith_c}. A precise proof of this fact is given in \cite[Prop.~3.1]{GR}. The proof of the equivalence is complete.
\end{proof}
We now come to the:
\begin{proof}[Proof of Proposition~\ref{lem:FDN_implies_EC}]
  Assume for contradiction that there exist $w\in W \setminus \{0\}$ and a subspace $H\leq\R^n$ with $\dim H=2$ such that $w\in\mathrm{im\,}\A[\xi]$ for all $\xi\in H\zeroset$. We abbreviate $\A^\dagger[\xi]=(\A^*[\xi]\A[\xi])^{-1}\A^*[\xi]$, which is such that $\A^\dagger[\xi]\A[\xi]=\id_V$ for $\xi\in\R^n\setminus\{0\}$. We record that $\A[\xi]\A^\dagger[\xi]$ is the orthogonal projection on $\mathrm{im\,}\A[\xi]$. 
Let $e_i\in H$, $i=1,2$, be linearly independent. We choose $v\in V$ such that $\A_0[\cdot]\not\equiv0$, where we define
\begin{align*}
\A_0[\xi]=v\cdot\A^\dagger[\xi]w \qquad\text{ for }\xi\in H.
\end{align*}
This is possible since $\A^\dagger[\xi]w$ is never $0$ by ellipticity. Since $H$ is a plane through 0 and $\A^\dagger[\cdot]$ is a $(-k)$-homogeneous rational function, it follows that $\A_0[\cdot]$ is a $(-k)$-homogeneous, rational function. Thus, we can write $\A_0[\xi]=P(\xi)/Q(\xi)$ for coprime, homogeneous polynomials $P$ and $Q$ {on $H$}. 

We recall from Proposition~\ref{prop:smith} that there exists an integer $d\geq k$ and a homogeneous linear differential operator $\mathbb{B}[D]$ such that
\begin{align*}
D^d=\mathbb{B}[D]\circ\A[D].
\end{align*}
In other words, all homogeneous polynomials of degree $d$ can be written as multipliers of the symbol map $\A[\cdot]$. It follows that we can find scalar operators $\mathbb{B}_i[D]$ such that
\begin{align*}
(\xi\cdot e_i)^{d}v^*=\mathbb{B}_i[\xi]\A[\xi],\quad i=1,2,
\end{align*}
which then implies
\begin{align*}
\mathbb{B}_i[\xi]w=\mathbb{B}_i[\xi]\A[\xi]\A^\dagger[\xi]w=(\xi\cdot e_i)^d(v\cdot \A^\dagger[\xi]w)=(\xi\cdot e_i)^d P(\xi)/Q(\xi).
\end{align*}
Therefore, $Q(\xi)\mathbb{B}_i[\xi] w=(\xi\cdot e_i)^dP(\xi)$ for all $\xi\in H$ from which it follows that $Q$ divides both $(\xi\cdot e_i)^d$, $i=1,2$. Therefore $Q$ is constant, which contradicts the $(-k)$-homogeneity of $\A_0[\cdot]\not\equiv0$.
\end{proof}
The converse of Proposition~\ref{lem:FDN_implies_EC} \emph{only} holds for first order operators, where the key fact we use is the linearity of the map $\xi\mapsto\A[\xi]$. A similar computation appeared in \cite[Lem.~3.5]{GR}.
\begin{lemma}\label{lem:smahoz}
	Let $\A[D]$ be a first order elliptic operator that satisfies \eqref{eq:n-2_canc_intro}. Then $\A[D]$ is $\C$-elliptic.
\end{lemma}
\begin{proof}
	Assume for contradiction that $\A[D]$ is elliptic, but not $\mathbb{C}$-elliptic, so that there exist $\xi_1,\,\xi_2\in \R^n\setminus\{0\}$ and $v_1,\,v_2\in V\setminus\{0\}$ such that
	\begin{align*}
	\A[\xi_1]v_1&=\A[\xi_2]v_2\\
	\A[\xi_1]v_2&=-\A[\xi_2]v_1.
	\end{align*}
	By ellipticity of $\A[D]$, we have that $\xi_1,\,\xi_2$ and $v_1,\,v_2$ are linearly independent. 
	We then write
	\begin{align*}
	\A[a\xi_1+b\xi_2](av_1+bv_2)&=a^2\A[\xi_1]v_1+ab\A[\xi_1]v_2+ab\A[\xi_2]v_1+b^2\A[\xi_2]v_2\\
	&=(a^2+b^2)\A[\xi_1]v_1,
	\end{align*}
	so that, by ellipticity of $\A[D]$,
	\begin{align*}
	\bigcap_{0\neq\zeta\in\mathrm{span}\{\xi_1,\xi_2\}}\mathrm{im\,}\A[\zeta]\ni\A[\xi_1]v_1\neq0,
	\end{align*}
	so that $\A[D]$ fails \eqref{eq:n-2_canc_intro}, which concludes the proof since $\dim \mathrm{span}\{\xi_1,\xi_2\}=2$.
\end{proof}
We now turn to some examples, the first of which demonstrates that for operators of order $k\geq 2$, the implication of ellipticity and strong cancellation by $\C$-ellipticity is strict in general:
\begin{example}\label{ex:ESCnotCell}
	Let $k,\,n,\,N\geq 2$. The operator
	\begin{align*}
		\A[D]u=D^{k-1}(\partial_1 u_1+\partial_2 u_2,\,\partial_2u_1-\partial_2u_2,\partial_iu_j)_{i,j\notin\{1,2\}}
	\end{align*}
	defined for $u\colon\R^n\rightarrow\R^N$ is elliptic, strongly cancelling, but not $\C$-elliptic.
\end{example}
\begin{proof}
	The fact that $\A[D]$ is elliptic, but not $\C$-elliptic follows from \cite[Counterex.~3.4]{GR} and Proposition~\ref{prop:smith}. To see that $\A[D]$ satisfies \eqref{eq:n-2_canc_intro}, we write $\mathbb{B}=D^{k-1}$. It is obvious that $\mathbb{B}[D]$ has finite dimensional nullspace, so $\mathbb{B}[D]$ is $\C$-elliptic by Proposition~\ref{prop:smith}. Furthermore, $\mathbb{B}[D]$ is strongly cancelling by Lemma~\ref{lem:FDN_implies_EC}. To conclude, let $H\leq\R^n$ have $\dim H=2$. Then
	\begin{align*}
		\bigcap_{\xi\in H\setminus\{0\}}\mathrm{im\,}\A[\xi]\subset\bigcap_{\xi\in H\setminus\{0\}}\mathrm{im\,}\mathbb{B}[\xi]=\{0\}.
	\end{align*}
	The proof is complete.
\end{proof}
Finally, we directly demonstrate the condition of strong cancellation (and its failure) for by now well-understood differential operators, cf.~\cite{BDG,VS,Raita18}.
\begin{example}[The symmetric gradient]\label{ex:symmetricgradient}
We here consider for $n\geq 2$ the first order differential operator 
\begin{align*}
\E u:=\tfrac{1}{2}(Du+Du^{\top}),\qquad u\colon\R^{n}\to\R^{n}, 
\end{align*}
where, in the situation of \eqref{eq:form}, $V=\R^{n}$ and $W=\R_{\sym}^{n\times n}$. The nullspace of $\E$ is given by the \emph{rigid deformations} $\mathscr{R}(\R^{n}):=\{x\mapsto Ax+b\colon\;A\in\R_{\mathrm{skew}}^{n\times n},\,b\in\R^{n}\}$, and so $\E$ is $\mathbb{C}$-elliptic (see also \cite[Ex.~2.2]{BDG}). Hence, by Proposition~\ref{lem:FDN_implies_EC}, $\E\,$ is also strongly cancelling. 
\end{example}
\begin{example}[The trace-free symmetric gradient]\label{ex:tracefreesymmetricgradient}
We augment Example~\ref{ex:symmetricgradient} by the trace-free symmetric gradient operator, i.e., 
\begin{align*}
\E^{D}u:=\E u - \tfrac{1}{n}\di(u)E_{n\times n},\qquad u\colon\R^{n}\to\R^{n},
\end{align*}
where $E_{n\times n}\in\R^{n\times n}$ denotes the $(n\times n)$-unit matrix. In the framework of \eqref{eq:form}, we have $k=1$, $V=\R^{n}$ and $W=\R_{\sym}^{n\times n}$. If $n\geq 3$, then the elements of its nullspace are given by the \emph{conformal Killing vectors}, cf. \cite{ReshKernel}, which are a subspace of $\mathscr{P}_{2}(\R^{n},\R^{n})$. Hence $\E^{D}$ has finite dimensional nullspace, thus is $\mathbb{C}$-elliptic by Proposition~\ref{prop:smith} and strongly cancelling by Proposition~\ref{lem:FDN_implies_EC}. If $n=2$, cancellation and   strong cancellation coincide. On the other hand, by \cite[Sec.~2.4, Ex.~(a)]{GR},  $\E^{D}$ is not cancelling for $n=2$, and thus not strongly cancelling two dimensions. 
\end{example}

\section{Codimension $s<1$: Proofs of Theorems ~\ref{thm:int} and~\ref{thm:main1}}\label{sec:EC}
In this section, we give the proof of Theorems~\ref{thm:int} and \ref{thm:main1}.  For ease of exposition, the sufficiency and necessity parts are separated and dealt with in Sections~\ref{sec:suff_EC} and \ref{sec:thm1.1}, respectively. 
\subsection{Sufficiency of ellipticity and cancellation}\label{sec:suff_EC}
We now prove that elliptic and cancelling operators satisfy multiplicative trace inequalities. Toward the proof of Theorem~\ref{thm:main1}, we require:
\begin{proposition}[{\cite[Thm.~8.3]{VS}}]\label{prop:TriebelLizorkin}
Let $n\geq 1$, $k\geq 1$ and $\A[D]$ be an elliptic and cancelling differential operator of the form \eqref{eq:form}. Then if $s\in (k-n,k)$, $1<p<\infty$ and $1\leq q \leq\infty$ are such that $\tfrac{1}{p}-\frac{s}{n}=1-\frac{k}{n}$, then there exists a constant $c>0$ such that 
\begin{align*}
\|u\|_{{\dot\trli}{_{p,q}^{s}}(\R^{n},V)}\leq c\|\A[D]u\|_{\lebe^{1}(\R^{n},W)}\qquad\text{for all}\;u\in\hold_{c}^{\infty}(\R^{n},V). 
\end{align*}
\end{proposition}

The proof of \ref{itm:thm1.1_b}$\implies$\ref{itm:thm1.1_a} in Theorem~\ref{thm:main1} follows as a particular case of the following:
\begin{proposition}\label{prop:suff_EC}
	Let $\A[D]$ as in \eqref{eq:form} be elliptic and canceling, $1\leq l\leq \min\{n-1,k\}$ be an integer, $0\leq s<l$, $q=\frac{n-s}{n-l}$, and $\frac{s}{l}\frac{n-l}{n-s}<\theta\leq 1$. Then there exists a constant $c>0$ such that
	\begin{align*}
		\|D^{k-l}u\|_{\lebe^q(\dif\mu)}\leq c\|\mu\|_{\lebe^{1,n-s}}^{1/q}\|D^{k-l}u\|_{\lebe^{\frac{n}{n-l}}(\dif\mathscr L^n)}^{1-\theta}\|\A u\|^\theta_{\lebe^1(\dif\mathscr L^n)}
	\end{align*}
	for all $u\in\hold^\infty_c(\R^n,V)$ and positive Radon measures $\mu$.
\end{proposition}
\begin{proof}
%Let $\A[D]$ be elliptic and cancelling of order $k$ on $\R^n$ and $s\frac{n-1}{n-s}<\theta\leq 1$. The proof of the multiplicative trace inequalities follows from the following sequence of estimates, explained below. Let $s\frac{n-1}{n-s}<\alpha<\theta$ and recall from the statement that $0\leq s<1$ and $q=\frac{n-s}{n-1}$.
Let $s\frac{n-l}{n-s}<\alpha <\theta l$. Using the semigroup property of Riesz potentials and {Adams}' trace theorem - cf. \eqref{eq:adams} or \cite[Thm.~7.2.1]{AH} - we obtain
\begin{align}\label{eq:batch1}
\begin{split}
	\|D^{k-l}u\|_{\lebe^q(\dif\mu)}&=\|I_\alpha I_{-\alpha}D^{k-l}u\|_{\lebe^q(\dif\mu)}\\
	&\leq c\|\mu\|^{1/q}_{\lebe^{1,n-s}}\|I_{-\alpha}D^{k-l}u\|_{\lebe^p(\dif\mathscr{L}^n)},
\end{split}
\end{align}
where $p=\frac{n}{n-l+\alpha}$. This step requires the restriction $s\frac{n-l}{n-s}<\alpha$, since we must have $p<q$. By the definition of the Riesz potentials and basic properties of the Fourier transforms of derivatives we obtain the pointwise equality
\begin{align}\label{eq:batch2}
\begin{split}
	\|I_{-\alpha}D^{k-l}u\|_{\lebe^p(\dif\mathscr{L}^n)}&= c\bigl\|\mathscr{F}^{-1}\bigl(|\xi|^\alpha \hat u(\xi)\otimes\xi^{\otimes(k-l)}\bigr)\bigr\|_{\lebe^p(\dif\mathscr{L}^n)}\\
	&\leq c\bigl\|\mathscr{F}^{-1}\left(|\xi|^{k-l+\alpha} \hat u(\xi)\right)\bigr\|_{\lebe^p(\dif\mathscr{L}^n)}\\
	&= c\|I_{-(k-l+\alpha)}u\|_{\lebe^p(\dif\mathscr{L}^n)},
\end{split}
\end{align}
where the second line follows from the H\"ormander-Mihlin multiplier theorem, since $1<p<\infty$ from the restrictions on $\alpha$ and the fact that
\begin{align*}
|\xi|^\alpha \hat u(\xi)\otimes\xi^{\otimes(k-l)}=\left[|\xi|^{k-l+\alpha} \hat u(\xi)\right]\otimes\left(\frac{\xi}{|\xi|}\right)^{\otimes(k-l)}.
\end{align*}
The third line in \eqref{eq:batch2} follows, again, by the definition of the Riesz potentials. By the isomorphism between Riesz potential spaces and certain homogeneous Triebel-Lizorkin spaces (see Lemma~\ref{lem:triebel_lizorkin}), we can further estimate
\begin{align}\label{eq:batch3}
\begin{split}
	\|I_{-(k-l+\alpha)}u\|_{\lebe^p(\dif\mathscr{L}^n)}&\leq c\|u\|_{\dot{\trli}{^{k-l+\alpha}_{p,2}(\R^n)}}\\
	&\leq c\|u\|^{1-\theta}_{\dot{\trli}{^{k-l}_{\frac{n}{n-l},2}(\R^n)}}\|u\|^\theta_{\dot{\trli}{^{k-l+\gamma}_{\frac{n}{n-l+\gamma},2}(\R^n)}},
\end{split}	%\\
	%&\leq c\|D^{k-1}u\|^{1-\theta}_{\lebe^{1^*}(\dif\mathscr{L}^n)}\|\A[D]u\|^\theta_{\lebe^1(\dif\mathscr{L}^n)},
\end{align}
where the second inequality follows with $\gamma=\frac{\alpha}{\theta}\in(0,l)$ by interpolation of homogeneous Triebel-Lizorkin spaces (see Lemma~\ref{lem:interpolation}). We proceed to estimate each term arising from \eqref{eq:batch3}. Firstly, using Lemma~\ref{lem:triebel_lizorkin}, we have that
\begin{align}\label{eq:batch4}
	\|u\|_{\dot{\trli}{^{k-l}_{\frac{n}{n-l},2}}}\leq c\|D^{k-l}u\|_{\lebe^{\frac{n}{n-l}}(\dif\mathscr{L}^n)}.
\end{align}
Finally, by Proposition~\ref{prop:TriebelLizorkin}, we have that 
\begin{align}\label{eq:batch5}
	\|u\|_{\dot{\trli}{^{k-l+\gamma}_{\frac{n}{n-l+\gamma},2}(\R^n)}}\leq c\|\A[D]u\|_{\lebe^1(\dif\mathscr{L}^n)},
\end{align}
which is compatible with the restriction $0<\gamma<l$.

The proof is then concluded by concatenating the estimates \eqref{eq:batch1}, \eqref{eq:batch2}, \eqref{eq:batch3}, \eqref{eq:batch4}, and \eqref{eq:batch5} in their order of appearance.
\end{proof}
Note that, if the codimension equals $s=l$ when $q=1$, then the method breaks down, e.g., since we require $1=\frac{s}{l}\frac{n-l}{n-s}<\theta\leq 1$. This is no coincidence, as can be seen from Section~\ref{sec:s=1}. Somewhat surprisingly, the only instance in which we can deal with the critical codimension case $s=l$ of Proposition~\ref{prop:suff_EC} is the limiting case $l=n$:
\begin{remark}
	Let $\A[D]$ be a $k$-th order elliptic and canceling operator on $\R^n$, with $k\geq n$. By the main results of the second author's recent work with \textsc{Skorobogatova} \cite{RaSk}, we have that for any $u\in\bv^\A(\R^n)$, we have that $D^{k-n}u$ is continuous, vanishing at infinity. Here we recall that $\bv^\A(\R^n)$ is the space of functions $u\in\sobo^{k-1,1}(\R^n,V)$ such that $\A[D]u\in\mathscr M(\R^n,W)$. Moreover, for any finite measure $\mu\in\mathscr M(\R^n)$ and any $u\in\bv^\A(\R^n)$, we have that the pairing $\langle \mu,|D^{k-n}u|\rangle$ is well-defined and the estimate
	\begin{align*}
		\|D^{k-n}u\|_{\lebe^1(\dif\mu)}\leq |\mu|(\R^n)\|D^{k-n}u\|_{\lebe^\infty(\R^n)}\leq c\|\mu\|_{\lebe^{1,0}}|\A[D]u|(\R^n),
	\end{align*}
	which seems to be a suitable, more general version of the estimate of Proposition~\ref{prop:suff_EC} in the limiting case $s=l=n$. The convention we make is $q=\frac{0}{0}=1$, which is consistent with the correct parameters for the case $s=l<n$. Also, it is not reasonable to expect a multiplicative inequality arising from interpolation in this case since $\sobo^{s,n/s}(\R^n)\hookrightarrow \hold_0(\R^n)$ only if $s=n$.
\end{remark}
%; in fact, as we will see in Section~\ref{sec:s=1} below, the constant $c=c(\A,s)>0$ from Theorem~\ref{thm:int}~\ref{itm:thm1.1_b} satisfies 
%\begin{align*}
%\lim_{s\nearrow 0}c(\A,s)=+\infty
%\end{align*}
%unless $\A[D]$ is $\mathbb{C}$-elliptic.

\subsection{Necessity of ellipticity and cancellation}\label{sec:thm1.1}
We first construct special regular sets of fractional dimension, which we will use as special choices of $\Sigma$ in the proof of necessity.
\begin{lemma}\label{lem:uniform_set}
Let $d\geq 1$ and $\alpha\in (0,d]$. Moreover, let $\mathcal{C}\subset\R^d$ be a double cone in the sense of Section~\ref{sec:notation}, having apex at $0$. Then there exists a subset $\Sigma\subset\mathcal{C}\cap B(0,1)$ and constants $0<m\leq M<\infty$ such that for all $0\leq r \leq 1$ there holds
\begin{align*}
mr^\alpha\leq \mathscr{H}^\alpha(B(0,r)\cap\Sigma)\leq Mr^{\alpha}.
\end{align*}
\end{lemma}
\begin{proof}
We recall from \cite{Hutchinson} that there exists a set $S\subset\R^d$ and constants $0<\lambda\leq\Lambda<\infty$ such that
\begin{align*}
\lambda\leq \liminf_{r\searrow 0}\dfrac{\mathscr{H}^\alpha(B(x,r)\cap S)}{r^\alpha}\leq \limsup_{r\searrow 0}\dfrac{\mathscr{H}^\alpha(B(x,r)\cap S)}{r^\alpha}\leq \Lambda
\end{align*}
holds for all $x\in S$. Since the 
Hausdorff measure is translation invariant, we can assume that $0\in S$, so that there exists $\varepsilon>0$ such that
\begin{align}\label{eq:S}
\dfrac{\lambda}{2}r^\alpha\leq \mathscr{H}^\alpha(B(0,r)\cap S)\leq 2\Lambda r^\alpha\qquad\text{for all }0\leq r\leq\varepsilon.
\end{align}
Applying a dilation by $\varepsilon^{-1}$ if necessary, we can assume that $\varepsilon=1$ (the fact that uniformly dilating by $\varepsilon^{-1}$ changes $\mathscr{H}^\alpha$-measure by a factor of $\varepsilon^{-\alpha}$ follows from the definition). 

In particular, the problem is solved if $d=1$, since in that case $\mathcal{C}\cap B(0,1)=(-1,1)$.

If $d\geq2$, note that it suffices to solve the problem for $d-1<\alpha\leq d$, otherwise we apply the following construction in a hyperplane of dimension $\lceil \alpha\rceil$.

Write now explicitly $\mathcal{C}=\R \mathcal{S}$, where $\mathcal{S}$ is a spherical cap in $\mathbb{S}^{n-1}$. Say that the unit vector $e\in\R^n$ is the centre of the sperical cap $\mathcal{S}$. Let $\theta\in(0,\pi/2)$ be the angle $e$ makes with $\partial\mathcal{C}$. Since we chose $\alpha>d-1$, \eqref{eq:S} holds with $S$ replaced by $S\setminus e^\perp$. Write now $H^+$ for the open half-space bounded by $e^\perp$ that contains $\mathcal{S}$ and let $\mathcal{C}^+=H^+\cap \mathcal{C}$. This geometrical situation is displayed in Figure~\ref{fig:cone}.

We will now define a bi-Lipschitz map $L^+$ between $H^+$ and $\mathcal{C}^+$ as follows: Let $x\in H^+$, so that $x\in |x|\mathbb{S}^{n-1}$. Consider the arc of angle $\theta_x$ in $|x|\mathbb{S}^{n-1}$ defined by $0,x,|x|e$. In this arc, there is a unique point $y$ (in $|x|\mathcal{S}$) such that the angle defined by $0,y,|x|e$ equals $\theta_y=2\theta_x\theta/\pi$. Define $L^+(x)=y$, so that the inverse of $L^+$ is given by $\theta_x=\pi\theta_y/(2\theta)$.

It is then clear by definition of $\mathscr{H}^\alpha$-measure that
\begin{align}\label{eq:L}
\mathscr{H}^{\alpha}\left(L^+(S\cap H^+)\cap B(0,r)\right)\sim \mathscr{H}^{\alpha}\left((S\cap H^+)\cap B(0,r)\right),
\end{align}
for $0\leq r\leq 1$, with constants given by the Lipschitz constants of $L^+$ and its inverse.

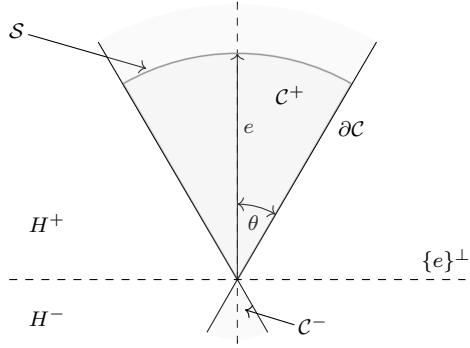
\begin{figure}
\begin{tikzpicture}
\draw (0,0) -- (-1.83,3.15);
\draw (0,0) -- (1.83,3.15);
\draw [thick,fill=gray!20,opacity=0.2](120:3)--(0,0)--(60:3) arc (60:120:3)--cycle;
\draw[->] (0,0) -- (0,3);
\node at (0.15,2) {\footnotesize{\text{$e$}}};
\draw [dashed] (0,-0.85) -- (0,3.75);
\draw [<->]  (60:1)   arc (60:90:1);
\node at (0.21,0.75) {\footnotesize{$\theta$}};
\draw [->] (-2.75,3.25) -- (-1.2,2.8);
\node at (-2.9,3.3) {\footnotesize{\text{$\mathcal{S}$}}};
\draw [-,gray]  (60:3)   arc (60:120:3);
\path [fill=gray!20,opacity=0.2](120:3.65)--(0,0)--(60:3.65) arc (60:120:3.65);
\node at (1.5,2) {\footnotesize{\text{$\partial\mathcal{C}$}}};
\draw (0,0) -- (0.4,-0.7);
\draw (0,0) -- (-0.4,-0.7);
\path [fill=gray!20,opacity=0.2](300:0.8)--(0,0)--(240:0.8) arc (240:300:0.8) -- cycle;
\draw [-,dashed] (-3,0) -- (3,0);
\node at (2.75,0.25) {\footnotesize{\text{$\{e\}^{\bot}$}}};
\node at (-2.5,0.75) {\footnotesize{\text{$H^{+}$}}};
\node at (0.7,2.45) {{\footnotesize{\text{$\mathcal{C}^{+}$}}}};
\node at (-2.5,-0.5) {\footnotesize{\text{$H^{-}$}}};
\node at (1,-0.6) {\footnotesize{\text{$\mathcal{C}^{-}$}}};
\draw[->] (0.75,-0.6) -- (0.1,-0.4);
\end{tikzpicture}
\caption{The geometrical situation in the proof of Lemma~\ref{lem:uniform_set}.}
\label{fig:cone}
\end{figure}

One then defines $L^-(x)=QL(Qx)$ for $x\in H^-=-H^+$, where $Q$ is the reflection in $e^\perp$. Finally, we have that $\Sigma=L^+(S\cap H^+)\cup L^{-}(S\cap H^{-})$ satisfies the assumptions of the Lemma: By \eqref{eq:S} and \eqref{eq:L} we have that
\begin{align*}
\mathscr{H}^\alpha\left(\Sigma\cap B(0,r)\right)\sim\mathscr{H}^\alpha\left(S\setminus e^\perp\cap B(0,r)\right)\sim r^\alpha,
\end{align*}
with implicit constants $m,\,M$ depending on $\lambda,\,\Lambda$, the Lipschitz constants of $L^+$ and its inverse, and $\varepsilon$. The proof is complete. 
\end{proof}
We next show that on the set $\Sigma$ thus constructed, non-zero continuous $(-\alpha)$-homogeneous functions are not $\mathscr{H}^\alpha$-integrable.
\begin{lemma}\label{lem:-alpha_hom}
Let $\Sigma$ be as given by Lemma~\ref{lem:uniform_set}. Then
\begin{align*}
\int_{\Sigma}|x|^{-\alpha}\dif\mathscr{H}^\alpha(x)=\infty.
\end{align*}
\end{lemma}
\begin{proof}
Suppose that $|\cdot|^\alpha\mres\Sigma$ is $\mathscr{H}^\alpha$-integrable. Then
\begin{align}\label{eq:plm}
	\int_{\Sigma\cap B(0,r)}|x|^\alpha\dif\mathscr{H}^\alpha(x)\rightarrow0
\end{align}
by the dominated convergence theorem. However, by Lemma~\ref{lem:uniform_set}, we have that the left hand side of \eqref{eq:plm} is bounded from below by $m>0$.
\end{proof}

\begin{proof}[Proof of \ref{itm:thm1.1_a}$\implies$\ref{itm:thm1.1_b} in Theorem~\ref{thm:main1}]
	We set $\theta=1$ in \ref{itm:thm1.1_a}.
	
\emph{Necessity of ellipticity}. We initially cover the first order case $k=1$. Assume that the estimate holds and $\A[D]$ is not elliptic. Without loss of generality, we can write that there exists $0\neq v\in V$ such that $\A[e_n]v=0$, where $(e_j)$ is an orthonormal basis of $\R^n$. We let $f(t)=|t|^{-\beta}$ for $t\in\R$ and $\beta=\frac{(1-s)(n-1)}{n-s}$, so that $f\in\lebe^1_{\locc}(\R)$. We then define $u_f(x)=f(x\cdot e_n)v$ for $x\in\R^n$, such that $\A[D]u_f=0$ in $\mathscr{D}^\prime(\R^n,V)$. To see that $u_f\in\lebe^1_{\locc}$, we write $x=(x^\prime,x_n)$ with respect to the basis $(e_j)_{j=1}^n$ and obtain
\begin{align*}
\int_{(-R,R)^n}|u_f|\dif\mathscr{L}^n&=\int_{(-R,R)^{n-1}}\int_{-R}^R f(x_n)|v|\dif x_n \dif x^\prime\\
&=|v|(2R)^{n-1}\int_{-R}^R f(t)\dif t<\infty.
\end{align*}
It thus follows by the Leibniz rule that $\tilde{u}_f=\rho u_f$ for $\rho\in\hold^\infty_c(\R^n)$ satisfies $\A[D]\tilde{u}_f\in\lebe^1(\R^n,W)$, so the right hand side of the assumed estimate is finite.

To proceed, we let $S\subset\R$ be as in Lemma \ref{lem:uniform_set} with $d=1$ and $\alpha=1-s$ and suppose that $\rho$ as above satisfies $\rho=1$ in $[0,1]^n$. Then, with $q=\frac{n-s}{n-1}$ and $\Sigma=[0,1]^{n-1}\times S$,
\begin{align*}
\int_{\Sigma}|\tilde{u}_f|^{q}\dif\mathscr{H}^{n-s}\geq \int_{(0,1)^{n-1}}\int_S f(x_n)^q|v|^q\dif\mathscr{H}^{\alpha}(x_n)\dif x^\prime=|v|^q\int_S|t|^{-\alpha}\dif\mathscr{H}^\alpha(t).
\end{align*}
To see that the last integral is infinite, one employs Lemma~\ref{lem:-alpha_hom}. To obtain a contradiction from the precise form of \ref{itm:thm1.1_a}, we note that the mollifications $\rho_{\varepsilon}*\widetilde{u}_{f}$ converge to $\widetilde{u}_{f}$ at all Lebesgue points of $\widetilde{u}_{f}$. Hence, should \ref{itm:thm1.1_a} hold, we obtain by Fatou's lemma the contradictory 
\begin{align}\label{eq:approximationargument}
\begin{split}
\int_{\Sigma}|\widetilde{u}_{f}|^{q}\dif\mathscr{H}^{n-s} & \leq \liminf_{\varepsilon\searrow 0}\int_{\Sigma}|\rho_{\varepsilon}*\widetilde{u}_{f}|^{q}\dif\mathscr{H}^{n-s}\\ 
& \leq c\|\mu\|_{\lebe^{1,n-s}(\R^{n})}^{\frac{n-1}{n-s}}\liminf_{\varepsilon\searrow 0}\|\A[D](\rho_{\varepsilon}*\widetilde{u}_{f})\|_{\lebe^{1}(\R^{n},W)} \\ & = c\|\mu\|_{\lebe^{1,n-s}(\R^{n})}^{\frac{n-1}{n-s}}\|\A[D](\widetilde{u}_{f})\|_{\lebe^{1}(\R^{n},W)}, 
\end{split}
\end{align}
so that \ref{itm:thm1.1_a} cannot hold indeed. 

We next assume that the order $k$ of $\A[D]$ is arbitrary and keep all relevant notation from the $k=1$ step. We now let $g(t)=|t|^{k-1-\beta}$ for $t\in\R$, which is clearly locally integrable, and let $u_g(x)=g(x_n)v$ for $x\in\R^n$, so that $u_g\in\lebe^1_{\locc}(\R^n,V)$, and, moreover, $\A[D]u_g=0$ in the sense of distributions. We also define 
\begin{align}\label{eq:utildedef}
\tilde{u}_g=\rho u_g,
\end{align}
so that $\tilde{u}_g=u_g$ in $(0,1)^n$. We next claim that $\A[D]\tilde{u}_g\in\lebe^1(\R^n,W)$, which is slightly more involved than in the first order case. We first note that $D^{k-1}u_g\in\lebe^1_{\locc}(\R^n,V\odot^{k-1}\R^n)$ by computing
\begin{align*}
D^{k-1}u_g(x)=\dfrac{\dif^{k-1} g}{\dif t^{k-1}}(x_n)v\otimes^{k-1}e_n=\frac{(\mathrm{sgn} (x_n))^{k-1}}{(k-1)!}f(x_n)v\otimes^{k-1}e_n
\end{align*}
and arguing as for the local integrability of $u_f$. By the \textsc{Deny-Lyons} Lemma \cite{DL}, we have that $D^ju_g$ is locally integrable for $0\leq j\leq k-1$. We then write
\begin{align*}
\A[D]\tilde{u}_g=\sum_{j=0}^{k-1}B_j[D^j u,D^{k-j}\rho],
\end{align*}
where $B_j$ are bilinear pairings depending on $\A[D]$ only. Since $\rho$ is smooth with compact support, it immediately follows that $\A[D]\tilde{u}_g\in\mathscr{M}(\R^n,W)$. With $\alpha=1-s$,
\begin{align*}
\int_\Sigma |D^{k-1}\tilde{u}_g|^q\dif\mathscr{H}^{n-s}=\left(\frac{|v\otimes^{k-1}e_n|}{(k-1)!}\right)^q\int_S f(t)\dif\mathscr{H}^\alpha(t)=\infty,
\end{align*}
which concludes the proof of necessity of ellipticity upon performing a mollification argument analogous to \eqref{eq:approximationargument}.

\emph{Necessity of cancellation}. Suppose that $\A[D]$ is elliptic and non-cancelling. By \cite[Lem.~2.5]{R}, there exists a map $u\in\hold^\infty(\R^n\setminus\{0\},V)\cap\lebe^1_{\locc}(\R^n,V)$ such that $\A[D]u=\delta_0w$ for some $0\neq w\in W$ and $D^{k-1}u$ is homogeneous of degree $1-n$. Since $w\neq 0$, $D^{k-1}u$ must be non-zero at a point $x_0\in\mathbb{S}^{n-1}$. By continuity, $D^{k-1}u$ must be bounded away from $0$ in a relatively open spherical neighbourhood $\mathcal{N}$ of $x_0$ in $\mathbb{S}^{n-1}$. By homogeneity, this implies, for some $c>0$, the bound
\begin{align*}
|D^{k-1}u(x)|^\frac{n-s}{n-1}\geq c|x|^{s-n}\quad\text{ for all }x\in \mathcal{C},
\end{align*}
where $\mathcal	{C}$ is the open double cone $\R_* \mathcal{N}$.

Let $\Sigma\subset\mathcal{C}$ be as in Lemma \ref{lem:uniform_set}. Lemma~\ref{lem:-alpha_hom} implies that $D^{k-1}u$ is not $\lebe^{\frac{n-s}{n-1}}$-integrable with respect to $\mu=\mathscr{H}^{n-s}\mres\Sigma\in\lebe^{1,n-s}$. This contradicts the assumed inequality and concludes the proof of this case, as the measure $\A[D]u$ is clearly bounded.
\end{proof} 

This completes the proof of necessity of ellipticity and cancellation for the trace inequality in Theorem~\ref{thm:main1}. Together with the sufficiency proof in Section~\ref{sec:suff_EC}, we have completed the proof of Theorem~\ref{thm:main1}.

\section{Codimension $s=1$: Proof of Theorem~\ref{thm:main2}}\label{sec:s=1}
In the case of critical codimension $s=1$, there is little hope so far to identify the $k$-th order operators $\A[D]$ for which inequalities
\begin{align*}
	\|D^{k-1}u\|_{\lebe^1(\dif \mu)}\leq c\|\mu\|_{\lebe^{1,n-1}}\|\A[D]u\|_{\lebe^1(\dif\mathscr{L}^n)}\quad\text{ for }u\in\hold^\infty_c(\R^n,V)
\end{align*}
hold. We thus restrict our attention to the purely $(n-1)$-dimensional case, by which we mean $\mu=\mathscr{H}^{n-1}\mres \Sigma$ for $\Sigma\leq \R^n$ with $\dim\Sigma=n-1$ (of course, if this is achieved, one likely can use the ideas to tackle the case $\mu=(f\mathscr{H}^{n-1}\mres S)$ for countably $(n-1)$-rectifiable $S\subset\R^n$ and $f\in\lebe^\infty(S;\dif\mathscr{H}^{n-1})$). Even with this restriction, the problem seems extremely challenging: In Section~\ref{sec:prop1.2} we give a necessary condition (strong cancellation); in Section~\ref{sec:ext} we give a sufficient condition ($\C$-ellipticity); finally, in Section~\ref{sec:pf_s=1}, we conclude the solution of the purely $(n-1)$-dimensional case for first order operators, $k=1$ (see Theorem~\ref{thm:main2}).
\subsection{Proof of Proposition~\ref{prop:int_tr_implies_En-2canc}}\label{sec:prop1.2} In view of Proposition~\ref{prop:int_tr_implies_En-2canc}, we first record the next
\begin{lemma}\label{lem:aux_nec_nu}
  Let $\A[D]$ as in \eqref{eq:form} be elliptic, $2\leq d\leq n$ be an integer, and $H\leq\R^n$ be a $d$-dimensional subspace of $\R^n$. Suppose that $w\in\mathrm{im\,}\A[\xi]$ for all $\xi\in H \setminus \{0\}$.

Then there exists a map $u\in\hold^\infty(\R^n\setminus H^\perp,V)$ and a $(1-d)$-homogeneous map $F\in\hold^\infty(H\setminus\{0\},V\odot^{k-1}\R^n)$ such that
\begin{align*}
\A[D]u=\left(\mathscr{H}^{n-d}\mres H^\perp\right)w\quad\text{and}\quad D^{k-1}u(g+h)=F(h)
\end{align*}
for all $h\in H$, $g\in H^\perp$.
\end{lemma}
\begin{proof}
We define, for $h\in H$, $g\in H^\perp$
\begin{align*}
u(h+g)&=\mathscr{F}^{-1}\left[\left(\mathscr{H}^{d}\mres H\right)\A^\dagger[\cdot]w\right](g+h)\\
&=c_n\int_{\R^n}e^{\imag (g+h)\cdot\xi}\A^\dagger[\xi]w\dif\,(\mathscr{H}^{d}\mres H)(\xi)\\
&=c_n\int_{H}e^{\imag h\cdot\xi}\A^\dagger[\xi]w\dif\mathscr{H}^{d}(\xi)\\
&=\mathscr{F}^{-1}_{H}(\A^{\dagger}[\cdot]w)(h),
\end{align*}
where $\A^\dagger[\xi]=(\A^*[\xi]\A[\xi])^{-1}\A^*[\xi]$ for $\xi\neq0$ and the right-hand side is a tempered distribution in $\hold^\infty(H^\perp\zeroset,V)$ by the proof of \cite[Lem.~2.1]{BVS}. More precisely, $\A[\cdot]w$ is a homogeneous distribution in $\R^n\setminus\{0\}$, which can be extended to a tempered distribution (see, e.g., \cite[Ch.~3,~Ch.~7]{HormI} for the general theory revolving around these facts). Using similar ideas, we can extrapolate that
\begin{align*}
D^{k-1}u(h+g)=\mathscr{F}^{-1}_{H}(\xi\mapsto\A^{\dagger}[\xi]w\otimes\xi^{\otimes(k-1)})(h)=:F(h),
\end{align*}
where $F$ has the required regularity and we can further use \cite[Lem.~2.1]{BVS} to see that $F$ is $(1-d)$-homogeneous.

To complete the proof, note that 
\begin{align*}
\widehat{u}(\xi)&=(\mathscr{H}^d\mres H)\A^\dagger[\xi]w\\
\widehat{\A[D]u}(\xi)&=(\mathscr{H}^d\mres H)\A[\xi]\A^\dagger[\xi]w.
\end{align*}
By recalling that $\A[\xi]\A^\dagger[\xi]$ is the orthogonal projection onto $\mathrm{im\,}\A[\xi]$, the conclusion follows.
\end{proof}
We can now conclude the:
\begin{proof}[Proof of Proposition~\ref{prop:int_tr_implies_En-2canc}]
To prove ellipticity of $\A[D]$, one can choose $g(t)=|t|^{k-\frac{3}{2}}$ in the proof of Theorem~\ref{thm:main1} and define $\tilde u_g$ in the same manner, cf.~\eqref{eq:utildedef}. It is then clear that $D^{k-1}\tilde u_g$ admits no trace on $\Sigma=e_n^\perp$.

If $\A[D]$ is elliptic, we assume that the embedding holds and that \eqref{eq:n-2_canc_intro} fails for some $2$-dimensional subspace $H\leq\R^n$ and $0\neq w\in \mathrm{im\,}\A[\xi]$ for all $\xi\in H\setminus\{0\}$. We let $u$ be as in Lemma~\ref{lem:aux_nec_nu} with $d=2$. Then 
\begin{align*}
D^{k-1}u(h+g)=F(h)=D^{k-1}\mathscr{F}^{-1}_H(\A^\dagger[\cdot]w)(h)\qquad\text{for }h\in H,\,g\in H^\perp
\end{align*}
is $(-1)$-homogenous and smooth in $H\setminus\{0\}$. Note that if $F\equiv0$, then $w\delta_0=\A[D]u=0$, so that there exists $\eta\in\mathbb{S}^{n-1}\cap H$ such that $F(\eta)\neq0$. By continuity of $F$ on $\mathbb{S}^{n-1}$ and homogeneity, we can assume that $\eta\neq\pm\nu$. %To see this, note that the operator $\A_H$, defined on $H$ from $V$ to $W$ by $\A_H[\xi]=\A[\xi]$ for $\xi\in H$, is elliptic and $\A_{H}\mathscr{F}^{-1}_H(\A^\dagger[\cdot]w)=\delta_0w$, so that $\A_H\mathscr{F}^{-1}_H(\A^\dagger[\cdot]w)=0$ in $H\setminus\{0\}$.

Up to a change of coordinate, we can assume that $\nu=e_1$ and $\eta=e_2$. %For points $x\in\R^n$ we write $x=(x_1,x_2,x^\prime)$. 
Let $Q$ be the unit cube in these coordinates. For a function $\rho\in\hold^\infty_c(\R^n)$ so that $\rho=1$ in $Q$, we have by the argument in the proof of Theorem~\ref{thm:main1} that $\rho u$ is admissible for the estimate with $|\A[D](\rho u)|(\R^n)<\infty$. We then have
\begin{align*}
\int_{Q\cap\Sigma}|D^{k-1}u|\dif\mathscr{H}^{n-1}=\int_0^1\int_{Q\cap H^\perp}|F(te_2)|\dif\mathscr{H}^{n-2}\dif t=|F(e_2)|\int_0^1\frac{\dif t}{t}=\infty,
\end{align*}
which concludes the proof.
\end{proof}
\subsection{Exterior traces}\label{sec:ext}
We next turn to the equivalence between \ref{it:int_k=1_est_strong} and \ref{it:int_k=1_Cell} in Theorem~\ref{thm:main2}, which we prove for operators of arbitrary order:
\begin{theorem}\label{thm:ext}
	Let $\A[D]$ be an operator as in \eqref{eq:form}. Then the following are equivalent: 
	\begin{enumerate}
		\item $\A[D]$ is \emph{$\mathbb{C}$-elliptic}.
		\item For every $(n-1)$-dimensional hyperplane $\Sigma\subset\R^{n}$ there exists a constant $c>0$ such that for all $u\in\hold_{c}^{\infty}(\R^{n},V)$ there holds 
		\begin{align*}
		\|D^{k-1}u\|_{\lebe^1(\Sigma;\dif\mathscr{H}^{n-1})} \leq c \|\A[D]u\|_{\lebe^1(\Sigma^+;\dif\mathscr{L}^{n})}. 
		\end{align*}
		Here $\Sigma^+$ is a half-space with boundary $\Sigma$.
	\end{enumerate}
\end{theorem}
Theorem~\ref{thm:ext} generalises the first order case as obtained by \textsc{Breit, Diening} and the first author \cite[Thm.~1.1]{BDG} in a natural way. The method of proof then is centered around extending a given function $u$ from $\Sigma^{+}$ to the entire $\R^{n}$ and   carefully employing $\lebe^{1}(\Sigma)$-Cauchy estimates for a suitable replacement of $u$ close to $\Sigma$. This replacement in turn is obtained by locally projecting $u$ onto $\ker(\A[D])$ on cubes close to $\Sigma$. At this stage, it is crucial to remark that $\mathbb{C}$-ellipticity is equivalent to $\A[D]$ having \emph{finite dimensional nullspace} (in $\mathscr{D}'(\R^{n},V)$) and consisting of polynomials of a \emph{bounded degree} (see Proposition~\ref{prop:smith}). Roughly speaking, since such polynomials form a finite dimensional vector space, we are thus in position to utilise \emph{inverse estimates}. These finally lead to the desired $\lebe^{1}(\Sigma)$-Cauchy property of the replacement sequence and hereafter the summability of the traces of $D^{k-1}u$ along $\Sigma$. Because the second part of Theorem~\ref{thm:ext} follows almost trivially from the first part, we can assert that \emph{$\mathbb{C}$-ellipticity is equivalent to boundary estimates, and moreover is sufficient for interior estimates}. 

In general, it is to be expected that boundary estimates are harder to obtain than interior estimates.  In fact, in the former case $\A[D]$ only provides control on $u$ from the interior of $\Sigma^{+}$, and so sequences of admissible maps might develop singularities along $\Sigma$. However, in the latter, interior trace case, $\A[D]$ provides control on $u$ on both sides of $\Sigma$. Hence, we expect that interior trace should be obtained under \emph{weaker} conditions than boundary traces (see Proposition~\ref{prop:int_tr_implies_En-2canc}).
\begin{proof}[Proof of Theorem~\ref{thm:ext}]
	Assume that $\A[D]$ is $\C$-elliptic. The trace inequality follows by an extension of the ideas in \cite{BDG}, coupled with the Poincar\'e-type inequality \cite[Prop.~4.2]{GR}. The presentation here is very streamlined and we refer the reader to \cite[Sec.~3-4]{BDG} and \cite[Sec.~4]{GR} for more detail, particularly, the introduction of basic tools used in the following arguments.
	
As shall be clear from the proof below, there is no loss of generality in assuming that $\Sigma=\R^{n-1}\times\{0\}$. We will use a coordinate notation $x=(x^\prime,x_n)$, with $x^\prime\in\R^{n-1}$, $x_n\in\R$. By a standard scaling argument, it suffices to prove that
	\begin{align*}
	\int_{\R^{n-1}}|D^{k-1}u|\dif x^\prime\leq c \int_{\R^{n}_+}|\A[D]u|\dif x\qquad\text{for all }u\in\hold^\infty_c(Q,V),
	\end{align*}
	where $Q=(-1,1)^n$ denotes the unit cube and $\R^{n}_+=\{x_n>0\}$. We will write $e_n=(0,1)$ for the unit normal to $\Sigma$.
	
        We now set the geometric scene: For $j=1,2,\ldots$, we write $\mathcal{S}^j$ for the strip $\{2^{-j}\leq x_n\leq 2^{-(j -1)}\}$. We cover $\mathcal{S}^j\cap Q$ with a collection of dyadic closed cubes $\{S_i^j\}_{i=1}^{2^{j(n-1)}}$ of sidelength $2^{-j}$ which have faces parallel to $Q$ and only intersect at faces. We write $Q_i^j$ for the open cubes that are obtained as follows: Translate $S_i^j$ by $-\frac{3}{2}\cdot2^{-j}e_n$ - so that its center lies on the hyperplane \(\Sigma\) - and dilate the interior of the resulting cube by a factor of $\frac{8}{7}$. In particular, the union $\mathcal{Q}^j=\bigcup_{i}Q_i^j$ covers $\Sigma\cap Q$ and the rectangle $R^j=\{x\in Q\colon -2^{-(j+1)}\leq x_n\leq 2^{-(j + 1)}\}$.
	
	For each $j\geq1$ we define partitions of unity $\{\varphi_i^j\}_i$ associated with $\{Q_i^j\}_i$ which satisfy $\sum_i\varphi_i^j=1$ in $R^j$ and $|D^l\varphi_i^j|\leq c2^{jl}$ for $l=0,1,\ldots k$. To localise near $\Sigma$, we will use functions $\rho_j\in\hold^\infty(\R^n)$ such that $\rho_j=1$ in $\mathcal{Q}^j$, $\rho_j(x)=0$ if $x_n\geq2^{-j}$, and $|D^l\rho_j|\leq c2^{jl}$ for $l\in\{0,1,\ldots, k\}$, $c>0$ being a universal constant.
	
	Let $\pi_C\colon\lebe^2(C)\rightarrow \ker(\A[D])$ denote the $\lebe^2$-orthogonal projection onto $\ker(\A[D])$, where $C\subset\R^{n}$ is a non-degenerate cube. Then we record in advance from \cite[Prop.~4.2]{GR} that there exists a constant $c>0$ independent of $u$ such that 
\begin{align}\label{eq:Poincare}
\sum_{l=1}^{k}\frac{\|D^{l}(u-\pi_{c}u)\|_{\lebe^{1}(C,V\odot^{l}\R^{n})}}{\ell(C)^{k-l}}\leq c\|\A[D]u\|_{\lebe^{1}(\R^{n},W)}, 
\end{align}	
where $\ell(C)$ is the sidelength of $C$. For the following, we abbreviate $\pi_i^j=\pi_{S_i^j}$. Before we proceed, let us note that building on \eqref{eq:Poincare} and the finite dimensional nullspace of $\A[D]$, it was established in \cite[Lem.~4.4]{GR} that there exists a constant $c=c(\A)>0$ such that if $Q_{i}^{j}$ and $Q_{m}^{j+1}$ have non-empty intersection, then there holds 
\begin{align}
\|D^{l}(\pi_{i}^{j}u-\pi_{m}^{j+1}u)\|_{\lebe^{1}(S_{i}^{j})}\leq c \ell(Q_{i}^{j})^{k-l}\|\A[D]u\|_{\lebe^{1}(\mathscr{N}(Q_{i}^{j}))}. 
\end{align}
Here, $\mathscr{N}(Q_{i}^{j})$ denotes the collection of all neighbouring cubes of $Q_{i}^{j}$, i.e., all cubes $Q_{i'}^{j}$ which have non-empty intersection with $Q_{i}^{j}$. By construction, for each $Q_{i}^{j}$ the number of cubes contained in $\mathscr{N}(Q_{i}^{j})$ is bounded by $n$ only.

	We then consider the trace approximation operator
	\begin{align*}
	T_ju=\rho_j\sum_i \varphi_i^j\pi_i^ju+(1-\rho_j)u.
	\end{align*}
	As in \cite[Sec.~4]{BDG}, we expect that $T_ju\rightarrow u$ is in several topologies, for which it is easier to first show that $(T_ju)_j$ is Cauchy in $\lebe^{1}(\Sigma)$. To this end, we record that
	\begin{align*}
	T_{j+1}u-T_ju&=(\rho_j-\rho_{j+1})\left(u-\sum_i \varphi_i^j\pi_i^ju\right)+\rho_{j+1}\left( \sum_m \varphi_m^{j+1}\pi_m^{j+1}u- \sum_i \varphi_i^j\pi_i^ju\right)\\
	&=(\rho_j-\rho_{j+1})\sum_i \varphi_i^j(u-\pi_i^ju)+\rho_{j+1}\left( \sum_{i,m} \varphi_m^{j+1}\varphi_i^j(\pi_m^{j+1}u-\pi_i^ju)\right),
	\end{align*}
        where in the second equality we used the summability to $1$ on $R^{j+ 1} \subset R^j$ of both partitions of unity. We will use this identity near $\Sigma$, where $\rho_j=1=\rho_{j+1}$, so the first term vanishes.
	
	The proof of the trace inequality is divided in a few steps:

	\emph{Step 1}. We show that $(D^{k-1}T_ju)_j$ is Cauchy in $\lebe^1(\Sigma)$. We note that in $\mathcal{Q}^{j+1}$ we have that
	\begin{align*}
	\mathbf{T}=D^{k-1}(T_{j+1}u-T_ju)&=\sum_{i,m} D^{k-1}[ \varphi_m^{j+1}\varphi_i^j(\pi_m^{j+1}u-\pi_i^ju)]\\
	&=\sum_{i,m}\sum_{l=0}^{k-1}\mathbf{B}_l[D^{k-l}(\varphi_m^{j+1}\varphi_i^j);D^l(\pi_m^{j+1}u-\pi_i^ju)],
	\end{align*}
	where $\mathbf{B}_l[\;{\cdot}\; ;\;{\cdot}\;]$ are bilinear pairings given by the Leibniz rule that depend only on $\A[D]$. We record that $\|D^{k-l}(\varphi_m^{j+1}\varphi_i^j)\|_{\lebe^{\infty}}\leq c2^{j(k-l)}$, a fact which also follows from the Leibniz rule. We then estimate 
	\begin{align*}
	\|\mathbf{T}\|_{\lebe^1(\Sigma)}&\leq c\sum_{i,m\colon Q_i^j\cap Q_m^{j+1}\neq\emptyset}\sum_{l=0}^{k-1}2^{j(k-l-1)}\|D^l(\pi_m^{j+1}u-\pi_i^ju)\|_{\lebe^1(\Sigma\cap Q_i^j)}\\
	&\leq c\sum_{i,m\colon Q_i^j\cap Q_m^{j+1}\neq\emptyset}\sum_{l=0}^{k-1}2^{j(k-l-n)}\|D^l(\pi_m^{j+1}u-\pi_i^ju)\|_{\lebe^\infty(\Sigma\cap Q_i^j)}\\
	&\leq c\sum_{i,m\colon Q_i^j\cap Q_m^{j+1}\neq\emptyset}\sum_{l=0}^{k-1}2^{j(k-l-n)}\|D^l(\pi_m^{j+1}u-\pi_i^ju)\|_{\lebe^\infty(\bar\Sigma^+\cap Q_i^j)}\\
	&\leq c\sum_{i,m\colon Q_i^j\cap Q_m^{j+1}\neq\emptyset}\sum_{l=0}^{k-1}2^{j(k-l)}\|D^l(\pi_m^{j+1}u-\pi_i^ju)\|_{\lebe^1(\Sigma^+\cap Q_i^j)}\\
	&\leq c\sum_{i,m\colon Q_i^j\cap Q_m^{j+1}\neq\emptyset}\sum_{l=0}^{k-1}2^{j(k-l)}\|D^l(\pi_m^{j+1}u-\pi_i^ju)\|_{\lebe^1(S_i^j)}\\
	&\leq c\|\A[D]u\|_{\lebe^1(\mathcal{S}^j)}.
	\end{align*}
For the fourth inequality, we used equivalence of all norms \emph{on finite dimensional spaces of fixed dimension} together with the suitable scaling, whereas in the last inequality we utilised the chain control lemma from \cite[Lem.~4.4]{GR}. We then have that
	\begin{align*}
	\|D^{k-1}(T_{j+s}u-T_j u)\|_{\lebe^1(\Sigma)}\leq \|\A[D]u\|_{\lebe^1(\{2^{-j-s}\leq x_n\leq 2^{-j}\})}.
	\end{align*}
	Since $\A[D]u\in\hold^\infty_c(Q,W)$, the claim of Step 1 follows.
	
	\emph{Step 2}. We show that $D^{k-1}T_ju\rightarrow D^{k-1}u$ in $\lebe^1(\Sigma)$. We have that $(D^{k-1}T_ju)_j$ converges in $\lebe^1(\Sigma)$, but it is not yet clear that the limit is $D^{k-1}u\restriction_\Sigma$. To prove that this is the case, we will show that $D^{k-1}T_ju\rightarrow u$ uniformly in $\bar\Sigma^+$. We write
	\begin{align*}
	D^{k-1}(u-T_ju)=\sum_i D^{k-1}\left[(\rho_j\varphi_i^j)(u-\pi_i^ju)\right].
	\end{align*}
	By local finiteness of the cover $\{Q_i^j\}_i$ and $|D^l(\rho_j\varphi^i_j)|\leq c2^{jl}$, we have that
	\begin{align*}
          \begin{split}
	\|D^{k-1}(u-T_ju)\|_{\lebe^\infty(\Sigma^+)}&\leq c\max_i \left\|D^{k-1}\left[(\rho_j\varphi_i^j)(u-\pi_i^ju)\right]\right\|_{\lebe^\infty(Q_i^j)}\\
	&\leq c\max_i\sum_{l=0}^{k-1}2^{j(k-l-1)}\|D^l(u-\pi_i^ju)\|_{\lebe^\infty(Q_i^j)}\\
        &\leq c\max_i\sum_{l=0}^{k-1}2^{j(k-l-1)}\left(\|D^l(u-\pi_{Q_i^j}u)\|_{\lebe^\infty(Q_i^j)}\right.\\[-1em]
        &\left.\hspace{12em} +\|D^l(\pi_{Q_i^j}u-\pi_{S_i^j})\|_{\lebe^\infty(Q_i^j)}\right)\\
	&\leq c2^{-j}\max_i\|\A[D]u\|_{\lebe^\infty(Q_i^j)}\leq c2^{-j}\|\A[D]u\|_{\lebe^\infty(\Sigma^+)}\rightarrow0,
          \end{split}
	\end{align*}
	which completes the proof of Step 2. In the fourth inequality we used \cite[Prop.~4.2]{GR} and \cite[Lem.~4.4]{GR}.
	
	\emph{Step 3}. It remains to prove the trace inequality. In view of Step 2, we have that
	\begin{align*}
	D^{k-1}u|_\Sigma=\text{$\lebe^1(\Sigma)$-}\!\!\lim_{s\rightarrow\infty} \sum_{j=2}^s D^{k-1}(T_{j+1}u-T_j u), 
	\end{align*}
	so that, by the inequality of Step 1, we have
	\begin{align*}
	\|D^{k-1}u\|_{\lebe^1(\Sigma)}&\leq \sum_{j\geq 2}\|D^{k-1}(T_{j+1}u-T_j u)\|_{\lebe^1(\Sigma)}\\ 
	& \leq\sum_{j\geq2}\|\A[D]u\|_{\lebe^1(\mathcal{S}^j)}
	=c\|\A[D]u\|_{\lebe^1(\Sigma^+)},
	\end{align*}
	which completes the proof of sufficiency of $\C$-ellipticity for the trace inequality.
	
	To complete the proof of the theorem, it remains to prove necessity of $\C$-ellipticity. Necessity of ellipticity follows by a repetition of the arguments in the proof of necessity of ellipticity for Proposition~\ref{prop:int_tr_implies_En-2canc}. Assume now  that $\A[D]$ is not $\mathbb{C}$-elliptic, so that there exist non-zero $\eta,\nu\in\R^n$ and $v\in V+\imag V$ such that $\A[\eta+\imag\nu]v=0$. There is no loss of generality in choosing $\Sigma^+=\{x\in\R^n\colon x\cdot\nu>0\}$. We choose a holomorphic branch of $\log\colon\mathbb{C}\setminus\imag(-\infty,0]\rightarrow\mathbb{C}$ and let $f_\varepsilon$ to be such that $f_\varepsilon^{(k-1)}(z)=(z+\varepsilon\imag)^{-1}$, so $f$ is holomorphic in $\mathbb{C}\setminus\imag(-\infty,-\varepsilon]$. Such a map exists by standard results of complex analysis. Define
	\begin{align*}
	u_\varepsilon(x)=f_\varepsilon(x\cdot\eta+\imag x\cdot\nu)v,
	\end{align*}
	so that $u_\varepsilon$ is smooth in the region where $f$ is defined. Though $u_\varepsilon$ is complex-valued, which we do not per se allow for, the following argument also holds for whichever of $\Re u_\varepsilon$ and $\Im u_\varepsilon$ satisfies the final estimate (one of them must).

	We note that $\A[D]u(x)=0$ for $x\cdot\nu>-\varepsilon$ by the proof of \cite[Prop.~3.1]{GR}. Let $\rho_\varepsilon\in\hold^\infty_c(\{x\cdot\nu>-\varepsilon\})$ be such that $\rho_\varepsilon=1$ in $B_1(0)\cap\Sigma^+$, $\rho_\varepsilon=0$ outside $B_2(0)\cap\Sigma^+$, and $|D^{l}\rho|\leq c$ for $l=0,\ldots k$. One the easily checks that $\rho_\varepsilon u_\varepsilon\in\hold^\infty_c(\R^n,V)$ and $\|\A[D](\rho_\varepsilon u_\varepsilon)\|_{\lebe^1(\Sigma^+)}\leq c$. Again by the proof of \cite[Prop.~3.1]{GR}, we have that
	\begin{align*}
	D^{k-1}u_\varepsilon(x)=f_\varepsilon^{(k-1)}(x\cdot\xi)v\otimes\xi^{\otimes(k-1)}\qquad\text{ for }x\in\Sigma^+,
	\end{align*}
	where $\xi=\eta+\imag\nu$. To check the failure of the estimate, one explicitly computes the limit $\lim_{\varepsilon\downarrow0}\|D^{k-1}(\rho_\varepsilon u_\varepsilon)\|_{\lebe^1({\Sigma\cap B_1(0)})}=\infty$.
\end{proof}

\subsection{Proof of Theorem~\ref{thm:main2}}\label{sec:pf_s=1} 
We can now collect the results proved so far in this Section to obtain our main result in the case of critical codimension $s=1$:
\begin{proof}[Proof of Theorem~\ref{thm:main2}]
We have that \ref{it:int_k=1_est} implies \ref{it:int_k=1_En-2C} from Proposition~\ref{prop:int_tr_implies_En-2canc}. The fact that \ref{it:int_k=1_En-2C} implies \ref{it:int_k=1_Cell} follows from Lemma~\ref{lem:smahoz}. By Theorem~\ref{thm:ext}, we have that \ref{it:int_k=1_Cell} is equivalent to \ref{it:int_k=1_est_strong}.
Finally, it is trivial to see that \ref{it:int_k=1_est_strong} implies \ref{it:int_k=1_est}. The proof is complete.
\end{proof}
To sum up what we covered in this section, we proved that for a trace inequality on hyperplanes the strong cancellation condition is necessary and $\C$-ellipticity is sufficient. In particular, for first order operators, the two notions are equivalent due to the linearity of the symbol map $\A[\cdot]$, hence the problem is solved. However, in general, $\C$-ellipticity is strictly stronger than ellipticity and strong cancellation, cf. Example~\ref{ex:ESCnotCell}. This motivates the following:
\begin{conjecture}\label{openproblem}
	Let $\A[D]$ be elliptic and strongly cancelling, of order $k$. Is it the case that for $(n-1)$-dimensional hyperplanes $\Sigma\leq \R^n$ and $u\in\hold^\infty_c(\R^n,V)$ we have
	\begin{align*}
	\|D^{k-1}u\|_{\lebe^1(\Sigma;\dif\mathscr{H}^{n-1})} \leq c \|\A[D]u\|_{\lebe^1(\R^n;\dif\mathscr{L}^{n})}\text{?} 
	\end{align*}
\end{conjecture}
\section{Function Space Implications}\label{sec:selected}
In this final section we aim at translating the foregoing results into the language of trace operators on function spaces, and particularly establish Theorem~\ref{thm:strict}. As specific novelties, implications for well-established function spaces such as $\bv$ and $\bd$ shall be given in Section~\ref{sec:examplesBVandBD}.
\subsection{Function spaces}\label{sec:BVAspaces}
To set up the required framework, let $\A[D]$ be of the form \eqref{eq:form} with $k=1$. Following \cite{BDG,GR}, we introduce for open sets $\Omega\subset\R^{n}$ the function spaces
\begin{align*}
\begin{split}
\sobo^{\A,p}(\Omega) & :=\big\{ u\in\lebe^{p}(\Omega,V)\colon\;\A[D]u\in\lebe^{p}(\Omega,W)\big\}, \\
\bv^{\A}(\Omega) & :=\big\{ u\in\lebe^{1}(\Omega,V)\colon\;\A[D]u\in\mathscr{M}(\Omega,W)\big\}
\end{split}
\end{align*}
where $1\leq p <\infty$. The norm on $\sobo^{\A,p}(\Omega)$ is canonically given by $\|u\|_{\sobo^{\A,p}(\Omega)}:=(\|u\|_{\lebe^{p}(\Omega,V)}^{p}+\|\A[D]u\|_{\lebe^{p}(\Omega,W)}^{p})^{1/p}$, whereas $\|u\|_{\bv^{\A}(\Omega)}:=\|u\|_{\lebe^{1}(\Omega,V)}+|\A[D]u|(\Omega)$ is the norm on $\bv^{\A}(\Omega)$ with $|\A[D]u|(\Omega)$ denoting the total variation of the $W$-valued measure $\A[D]u$.

As in the $\bv$-case, the norm topology on $\bv^{\A}$  is much too strong for many applications; so, for instance, elements in $\bv^{\A}(\Omega)$ cannot be approximated by maps in $\hold^{\infty}(\Omega,V)\cap\bv^{\A}(\Omega)$ in the norm topology. Following \cite{AFP} in the $\bv$-case, we say that a sequence $(u_{j})\subset\bv^{\A}(\Omega)$ converges to $u\in\bv^{\A}(\Omega)$ in the \emph{weak*-sense} provided $u_{j}\to u$ in $\lebe^{1}(\Omega,V)$ and $\A[D]u_{j}\stackrel{*}{\rightharpoonup} \A[D]u$ in $\mathscr{M}(\Omega,W)$. We moreover say that $u_{j}\to u$ \emph{strictly} (or $\A$\emph{-strictly}) in $\bv^{\A}(\Omega)$ provided that, in addition, $|\A[D]u_j|(\Omega)\rightarrow|\A[D]u|(\Omega)$ as $j\to\infty$. Let us note that if $u_{j}\to u$ $\A$-strictly, then $\A u_{j} \to \A u$ strictly in the sense of finite, $W$-valued Radon measures on $\Omega$. By routine means (also see \cite[Thm.~2.8]{BDG}), we obtain that for all $u\in\bv^{\A}(\R^{n})$ there holds 
\begin{align}\label{eq:strictapproximation}
\rho_{\varepsilon}*u \to u\qquad\text{$\A$-strictly in $\bv^{\A}(\R^{n})$}\;\;\text{as}\;\varepsilon\searrow 0, 
\end{align}
where  $\rho\in\hold_{c}^{\infty}(\ball(0,1),[0,1])$ is a radially symmetric mollifier with $\|\rho\|_{\lebe^{1}(\R^{n})}=1$ and $\rho_{\varepsilon}(x):=\varepsilon^{-n}\rho(\tfrac{x}{\varepsilon})$ its $\varepsilon$-rescaled variant. 

\subsection{Trace embeddings for maps of bounded $\A$-variation}
Let $n\geq 2$. Our first concern in this section is to establish that, for elliptic and cancelling operators $\A[D]$ of the form \eqref{eq:form} and measures $\mu\in\lebe^{1,n-s}(\R^{n})$ with $0\leq s<1$, $\mu$-traces can be assigned to $D^{k-1}u$ for all $u\in\bv^{\A}(\R^{n})$. For this purpose, it suffices to suppose that $\A[D]$ is a first order operator and to consequently establish the relevant assertions on $\mu$-traces for $u$, and we shall do so in the sequel. The situation for $\sobo^{\A,p}(\R^{n})$ is considerably easier, and we record it here for completeness: 
\begin{lemma}\label{lem:WAptracedefine}
Let $\A[D]$ be a first order elliptic operator of the form \eqref{eq:form}, and let $1<p<n$, $0\leq s < p$. Then there exists a norm-continuous linear trace operator $\trace_{\mu}\colon\sobo^{\A,p}(\R^{n})\to \lebe^{q}(\R^{n},V;\dif\mu)$, where $q=p\frac{n-s}{n-p}$. In particular, $\trace_{\mu}(\varphi)=\varphi|_{\spt(\mu)}$ for all $\varphi\in\sobo^{\A,p}(\R^{n})\cap\hold(\R^{n},V)$. 
\end{lemma}
\begin{proof} 
Let $\A[D]$ be elliptic. From the proof of Lemma~\ref{lem:aux_nec_nu}, we recall the notation $\A^{\dagger}[\xi]:=(\A^{*}[\xi]\A[\xi])^{-1}\A^{*}[\xi]\in\lin(W,V)$ for $\xi\in\R^{n}\setminus\{0\}$. We note that for each $\alpha\in\mathbb{N}_{0}^{n}$ with $|\alpha|=1$, the map $m_{\A,\alpha}\colon \xi\mapsto\xi^{\alpha}\A^{\dagger}[\xi]$ is homogeneous of degree zero and belongs to $\hold^{\infty}(\R^{n}\setminus\{0\},\lin(W,V))$. Hence, by the {H\"{o}rmander-Mihlin} multiplier theorem, the corresponding Fourier multiplication operator $\varphi\mapsto\mathscr{F}^{-1}(m_{\A}(\xi)(\mathscr{F}\varphi)(\xi))$, originally defined for $\varphi\in\hold_{c}^{\infty}(\R^{n},W)$, extends to a bounded linear operator $M_{\A,\alpha}\colon\lebe^{p}(\R^{n},W)\to\lebe^{p}(\R^{n},V)$. Since $M_{\A,\alpha}(\A[D]u)=c\partial^{\alpha}u$ for some $c\in\mathbb{C}$ and all $u\in\hold_{c}^{\infty}(\R^{n},V)$, we conclude by a routine approximation argument that $\sobo^{\A,p}(\R^{n})\simeq\sobo^{1,p}(\R^{n},V)$. From here, the conclusion of the theorem, which is well-known for Sobolev spaces $\sobo^{1,p}(\R^{n})$, is complete.
\end{proof}
As for the critical codimension case $s=p\in(1,n)$, we may utilise the recent result of \textsc{ Korobkov} and \textsc{Kristensen} \cite{KorKri}, which states that $\sobo^{1,(p,1)}(\R^{n},V)\hookrightarrow\lebe^p(\dif\mu)$ for $\mu\in\lebe^{1,n-p}(\R^{n})$. The norm continuity of the embedding $\sobo^{\A,(p,1)}(\R^{n})\hookrightarrow\lebe^p(\dif\mu)$ with the corresponding $\A$-Sobolev-Lorentz space $\sobo^{\A,(p,1)}$ then follows by boundedness of singular integrals on Lorentz spaces (see, e.g., \cite[Thm.~3.14]{EKS}). However, in our study of the limit case $s=p=1$ of \eqref{eq:main_ineq}, the Lorentz refinement is implicit as $\lebe^{(1,1)}\simeq\lebe^1$.

In the $\bv$-case, cf. \textsc{Ziemer} \cite[Thm.~5.13.1]{Ziemer}, the matter of assigning traces to $u$ on lower dimensional subsets is essentially reduced to the coarea formula for $\bv$-functions. The lack of a suitable version of the coarea formula in the $\A[D]$-framework forces us to argue differently:
\begin{proposition}\label{prop:tracedefine}
Let $\A[D]$ be a first order, elliptic and cancelling operator of the form \eqref{eq:form} and let $0\leq s < 1$. Then, for any $\mu\in\lebe^{1,n-s}(\R^{n})$, there exists a bounded linear trace operator 
\begin{align}
\trace_{\mu}\colon\bv^{\A}(\R^{n})\to \lebe^{\frac{n-s}{n-1}}(\R^{n};\dif\mu).
\end{align}
In particular, $\trace_{\mu}(\varphi)=\varphi|_{\spt(\mu)}$ for all $u\in\bv^{\A}(\R^{n})\cap\hold(\R^{n},V)$. 
\end{proposition}

\begin{proof}
Let $n\geq 2$ be fixed and suppose that $0<s<1$; for $s=0$ the claim is trivial. Given $0<\theta<1$, we have $\bv^{\A}\hookrightarrow\sobo^{\theta,p(\theta)}(\R^{n},V)$ for $p(\theta)=\frac{n}{n-1+\theta}$. In fact, given $u\in\bv^{\A}(\R^{n})$, pick $(u_{j})\subset\hold_{c}^{\infty}(\R^{n},V)$ such that $u_{j}\to u$ $\A$-strictly in $\bv^{\A}(\R^{n})$. For a non-relabeled subsequence, we then obtain $u_{j}\to u$ $\mathscr{L}^{n}$-a.e., and thus by Fatou's lemma , 
\begin{align*}
[u]_{\sobo^{\theta,p(\theta)}(\R^{n},V)}\leq \liminf_{j\to\infty}[u_{j}]_{\sobo^{\theta,p(\theta)}(\R^{n},V)}\leq \liminf_{j\to\infty}\|\A[D]u_{j}\|_{\lebe^{1}(\R^{n},W)}=|\A[D]u|(\R^{n}). 
\end{align*}
At this stage, we choose $s<\theta<1$. Then by Lemma~\ref{lem:Triebeltrace}, for every $\varphi\in\sobo^{\theta,p(\theta)}(\R^{n},V)$ $\mu$-a.e. $x\in\R^{n}$ is a Lebesgue point for $u$. We then define 
\begin{align*}
\trace_{\mu}(u)(x):=u^{*}(x)\qquad\text{for $\mu$-a.e. $x\in\R^{n}$}, 
\end{align*}
where $u^{*}$ denotes the precise representative of $u$ as usual. Now let $u_{\varepsilon}:=\rho_{\varepsilon}*u$. Then $u_{\varepsilon}(x)\to u^{*}(x)$ as $\varepsilon\searrow 0$ for all Lebesgue points $x$ of $u$. In consequence, we obtain by Fatou's lemma 
\begin{align*}
\|\trace_{\mu}u & \|_{\lebe^{\frac{n-s}{n-1}}(\R^{n};\dif\mu)} = \|u^{*}\|_{\lebe^{\frac{n-s}{n-1}}(\R^{n};\dif\mu)}\leq \liminf_{\varepsilon\searrow 0}\|u_{\varepsilon}\|_{\lebe^{\frac{n-s}{n-1}}(\R^{n};\dif\mu)}\\ 
& \leq c\liminf_{\varepsilon\searrow 0}\|\mu\|_{\lebe^{1,n-s}(\R^{n})}^{\frac{n-1}{n-s}}\|\A[D]u_{\varepsilon}\|_{\lebe^{1}(\R^{n},W)} \leq c\|\mu\|_{\lebe^{1,n-s}(\R^{n})}^{\frac{n-1}{n-s}}|\A[D]u|(\R^{n}). 
\end{align*}
The proof is complete. 
\end{proof}
\begin{remark}[Homogeneous spaces]
By an inexpensive modification of Lemma~\ref{lem:WAptracedefine} or Proposition~\ref{prop:tracedefine}, it is possible to set up an analogous trace theory for the corresponding \emph{homogeneous spaces}; here we confine to $p=1$. Namely, letting $\|v\|_{{\dot{\sobo}}{^{\A,1}}(\R^{n})}:=\|\A[D]v\|_{\lebe^{1}(\R^{n},W)}$ for $v\in\hold_{c}^{\infty}(\R^{n},V)$, we define ${\dot{\sobo}}{^{\A,1}}(\R^{n})$ to be the closure of the space $\hold_{c}^{\infty}(\R^{n},V)$ for this norm. In consequence, letting $\mu$ be as in Proposition~\ref{prop:tracedefine}, we obtain the existence of a norm continuous trace operator $\trace_{\mu}\colon{\dot{\sobo}}{^{\A,1}}(\R^{n})\to\lebe^{\frac{n-s}{n-1}}(\R^{n},V;\dif\mu)$ provided $\A[D]$ is elliptic and cancelling. A similar result holds for ${\dot{\bv}}{^{\A}}(\R^{n})$, which we define here as those $u\in\lebe^{\frac{n}{n-1}}(\R^{n},V)$ for which $\A[D]u\in\mathscr{M}(\R^{n},W)$, being normed by $\|u\|_{{\dot{\bv}}{^{\A}}(\R^{n})}:=|\A[D]u|(\R^{n})$.  
\end{remark}
The preceding Proposition~\ref{prop:tracedefine} does not remain valid for codimensions $s=1$, the reason essentially being the failure of Theorem~\ref{thm:main1} for this choice of $s$. Even though we could treat the situation for more general measures\footnote{Namely, by an adaptation of Theorem~\ref{thm:ext} in the spirit of \cite{BDG}, for measures $\mu$ absolutely continuous for restrictions of $\mathscr{H}^{n-1}$ to the boundaries of a wide class of NTA domains.}, we stick to the particular case of halfspaces as the underlying difficulties are already visible here.
\begin{proposition}\label{prop:traces=1}
Let $\Sigma\subset\R^{n}$ be an affine hyperplane and let $\mu\in\lebe^{1,n-1}(\R^{n})$. Then there exists a norm continuous linear trace operator 
\begin{align}
\trace_{\mu{\tiny \mres}\Sigma}\colon\bv^{\A}(\R^{n})\to \lebe^{1}(\R^{n},V;\dif\,(\mu\mres\Sigma)).
\end{align}
In particular, $\trace_{\mu{\tiny \mres}\Sigma}(\varphi)=\varphi|_{\Sigma}$ for all $u\in\bv^{\A}(\R^{n})\cap\hold(\R^{n},V)$. 
\end{proposition}
\begin{proof}
Since $\mu\in\lebe^{1,n-1}(\R^{n})$, $\mu\mres\Sigma\ll\mathscr{H}^{n-1}\mres\Sigma$, and so it suffices to prove the proposition for $\mu=\mathscr{H}^{n-1}\mres\Sigma$. We aim to construct the \emph{interior} trace by \emph{exterior} traces, and for this we firstly consider the spaces $\sobo^{\A,1}(\R^{n})$. For $\Sigma$ as in the proposition, denote $\Sigma^{+}$ and $\Sigma^{-}$ the open halfspaces determined (uniquely up to a change of $\pm$) by $\Sigma$. 

Given $v\in\sobo^{\A,1}(\Sigma^{+})$, we employ \cite[Thm.~4.1]{GR} to boundedly extend $v$ to $\overline{v}\in\sobo^{\A,1}(\R^{n})$, and then choose $(v_{j})\subset\hold(\R^{n},V)\cap\sobo^{\A,1}(\R^{n})$ such that $v_{j}\to \overline{v}$ in the norm topology of $\sobo^{\A,1}(\R^{n})$. By Theorem~\ref{thm:ext}, we have that $\|v_{j}-v_{l}\|_{\lebe^{1}(\Sigma,V;\dif\mathscr{H}^{n-1})}\leq c_{\Sigma^{+}}\|\A[D](v_{j}-v_{l})\|_{\lebe^{1}(\Sigma^{+},W;\dif\mathscr{L}^{n})}$, and analogous for $\Sigma^{-}$. Hence $(v_{j}|_{\Sigma})$ is Cauchy in $\lebe^{1}(\Sigma,V;\dif\mathscr{H}^{n-1})$, and thus converges to some $\trace_{\Sigma}^{+}(v)\in\lebe^{1}(\Sigma,V;\dif\mathscr{H}^{n-1})$. A routine argument yields that $\trace_{\Sigma}^{+}(v)$ is independent of the particular approximation sequence $(v_{j})$, and so the mapping $\sobo^{\A,1}(\Sigma^{+})\ni v\mapsto \trace_{\Sigma}^{+}(v)\in\lebe^{1}(\Sigma,V;\dif\mathscr{H}^{n-1})$ is well-defined. 

For $\bv^{\A}$-maps, the situation is a bit more delicate as the topology of $\A$-strict convergence does not admit the same approximation argument. We first establish that there exists a norm continuous, linear exterior trace operator for $\bv^{\A}(\Sigma^{\pm})$ each. Thus let $v\in\bv^{\A}(\Sigma^{+})$, and choose a sequence $(v_{j})\subset\hold(\Sigma^{+},V)\cap\bv^{\A}(\Sigma^{+})$ such that $v_{j}\to v$ $\A$-strictly in $\bv^{\A}(\Sigma^{+})$. Given $\varepsilon>0$, let $\varphi_{\varepsilon}\in\hold(\R^{n},[0,1])$ be a cut-off function with $\varphi_{\varepsilon}= 1$ in $\overline{U_{\varepsilon}^{+}}$, where $U_{\varepsilon}^{+}:=\{x\in\Sigma^{+}\colon\;0<\dist(x,\Sigma)\leq\varepsilon\}$, $\varphi_{\varepsilon}\equiv 0$ outside $\{x\in\Sigma\colon\;\dist(x,\Sigma)>2\varepsilon\}$ and $|\nabla\varphi_{\varepsilon}|\leq \tfrac{2}{\varepsilon}$. Then there holds for all $j,l\in\mathbb{N}$
\begin{align}\label{eq:traceestimatedefineBVA}
\begin{split}
\|\trace_{\Sigma}^{+}&(v_{j}-v_{l})\|_{\lebe^{1}(\Sigma,V;\dif\mathscr{H}^{n-1})}  = \|\trace_{\Sigma}^{+}(\varphi_{\varepsilon}(v_{j}-v_{l}))\|_{\lebe^{1}(\Sigma^{+},V;\dif\mathscr{H}^{n-1})}\\
& \leq c\|\A[D](\varphi_{\varepsilon}(v_{j}-v_{l}))\|_{\lebe^{1}(\Sigma^{+},W)} \\
& \leq c\|\varphi_{\varepsilon}\A[D](v_{j}-v_{l})\|_{\lebe^{1}(\Sigma^{\pm},W)} + \frac{c}{\varepsilon}\|v_{j}-v_{l}\|_{\lebe^{1}(\Sigma^{\pm},V)}\big)\\
& \leq c\big(|\A[D]v_{j}|(U_{2\varepsilon}^{+})+ |\A[D]v_{l}|(U_{2\varepsilon}^{+}) + \frac{1}{\varepsilon}\|v_{j}-v_{l}\|_{\lebe^{1}(\Sigma^{+},V)}\Big),
\end{split}
\end{align}
where $\trace_{\Sigma^{+}}$ is the exterior trace operator on $\sobo^{\A,1}(\Sigma^{+})$, in turn being well-defined by the first part of the proof. We now send $j,l\to\infty$ to obtain  
\begin{align*}
\lim_{j,l\to\infty}
\|\trace_{\Sigma^{+}}(v_{j}-v_{l})\|_{\lebe^{1}(\Sigma,V;\dif\mathscr{H}^{n-1})}\leq c\int_{\R^{n-1}\times (0,2\varepsilon)}\dif|\A[D]u|, 
\end{align*}
and then send $\varepsilon\searrow 0$ to obtain that $(\trace_{\Sigma^{+}}(v_{j}))$ is Cauchy in $\lebe^{1}(\Sigma,V;\dif\mathscr{H}^{n-1})$. As above, this easily implies the existence of a norm-continuous, linear exterior trace operator $\trace_{\Sigma}^{+}\colon\bv^{\A}(\Sigma^{+})\to\lebe^{1}(\Sigma,V;\dif\mathscr{H}^{n-1})$, and analogously $\trace_{\Sigma}^{-}\colon\bv^{\A}(\Sigma^{-})\to\lebe^{1}(\Sigma,V;\dif\mathscr{H}^{n-1})$. We now define for $u\in\bv^{\A}(\R^{n})$
\begin{align*}
\trace_{\Sigma}(u):=\frac{1}{2}(\trace_{\Sigma}^{+}(u|_{\Sigma^{+}})+\trace_{\Sigma}^{-}(u|_{\Sigma^{-}})), 
\end{align*}
and it is easily seen that this trace operator matches the properties as asserted. The proof is complete. 
\end{proof}
Two remarks are in order. 
\begin{remark}
In the preceding proof, we have argued by different smooth approximations in $\bv^{\A}(\Sigma^{+})$ and $\bv^{\A}(\Sigma^{-})$, respectively. Note that we cannot argue by global smooth approximations in the sense that we may not take $(v_{j})\subset\hold(\R^{n},V)\cap\bv^{\A}(\R^{n})$ such that $v_{j}\to v$ $\A$-strictly in $\bv^{\A}(\R^{n})$ in \eqref{eq:traceestimatedefineBVA}. In this situation, terms of the form $|\A[D]v_{j}|(\overline{U}_{2\varepsilon})$ appear on the very right hand side of \eqref{eq:traceestimatedefineBVA}, and these terms do \emph{not} vanish as $\varepsilon\searrow 0$. Essentially, this is a consequence of the fact that if $v_{j}\to v$ $\A$-strictly in $\bv^{\A}(\R^{n})$, then in general it \emph{does not} follows that the restrictions $v_{j}|_{\Sigma^{+}}$ converge $\A$-strictly to $v|_{\Sigma^{+}}$ in $\bv^{\A}(\Sigma^{+})$. A scenario when the $\A$-strict convergence of the restrictions can be obtained is discussed in step 2 of the proof of Proposition~\ref{prop:continuitySobolevembedding} below. 
\end{remark}
\begin{remark}[Exterior versus interior traces]
As directly extractable from the above proof, there exists a norm continuous, linear exterior trace operator $\trace\colon\bv^{\A}(\Sigma^{+})\to\lebe^{1}(\Sigma,V;\dif\mathscr{H}^{n-1})$. This operator is constructed as the $\A$-strictly continuous extension of the trace operator on $\sobo^{\A,1}(\Sigma^+)$ to $\bv^{\A}(\Sigma^+)$. Note that this \emph{does not} imply strict continuity of the \emph{interior} trace operator from Proposition~\ref{prop:tracedefine}. This is due to the fact that if $u_{j}\to u$ $\A$-strictly in $\bv^{\A}(\R^{n})$, then we \emph{do not} have $u_{j}|_{\Sigma^{\pm}}\to u|_{\Sigma^{\pm}}$ $\A$-strictly in $\bv^{\A}(\Sigma^{\pm})$ each; see Theorem~\ref{thm:inttracesBV} and the discussion afterwards for the failure of codimension one interior trace operators on $\bv^{\A}(\R^{n})$. 
\end{remark}
\subsection{Continuity of trace operators}
Proposition~\ref{prop:tracedefine} and~\ref{prop:traces=1} allow to define trace operators in two different settings which are \emph{continuous for the norm topology} on $\bv^{\A}$ each. As argued above, the norm topology is too weak for various applications, and so we now study the continuity of the trace operators with respect to $\A$-strict convergence. We begin with the following Proposition~\ref{prop:continuitySobolevembedding} in the spirit of \cite[Prop.~3.7]{RS}, where the situation for the $\bv$-case was considered. Then we boost it for codimension $0\leq s<1$-measures by use of the multiplicative trace inequality of Theorem~\ref{thm:main1}, thereby completing the proof of Theorem~\ref{thm:strict}. For its proof, we require the following fact which follows, e.g., from \cite[Thm.~3]{RW} or from the recent paper \cite{ARDPHR}.
\begin{lemma}\label{lem:nonatomic}
Let $\A[D]$ be an elliptic and cancelling operator of the form \eqref{eq:form}. Let $u\in\lebe_{\locc}^{1}(\R^{n},V)$. If $\A[D]u$ is a measure, then it is non-atomic.
\end{lemma}
\begin{proposition}\label{prop:continuitySobolevembedding}
Let $\A[D]$ be a first order, elliptic and cancelling differential operator of the form \eqref{eq:form}. Then the embedding 
\begin{align}
\bv^{\A}(\R^{n})\hookrightarrow \lebe^{\frac{n}{n-1}}(\R^{n},V)
\end{align}
is continuous for the $\A$-strict topology: If $u,u_{1},u_{2},...\in\bv^{\A}(\R^{n})$ are such that $u_{j}\to u$ $\A$-strictly in $\bv^{\A}(\R^{n})$, then $u_{j}\to u$ strongly in $\lebe^{\frac{n}{n-1}}(\R^{n},V)$ as $j\to\infty$.
\end{proposition}
\begin{proof}
The proof evolves in two steps. First, we establish that $\bv^{\A}(\ball(0,R))\hookrightarrow\lebe^{\frac{n}{n-1}}(\ball(0,R),V)$ on open and bounded Lipschitz domains, and secondly pass to the entire space $\R^{n}$ by an approximation argument. For the following, let $u,u_{1},u_{2},...\in\bv^{\A}(\R^{n})$ such that $u_{j}\to u$ $\A$-strictly in $\bv^{\A}(\R^{n})$ as $j\to\infty$. 

\emph{Step 1. An intermediate claim for bounded domains.} Our first aim is to show that  
\begin{align}\label{eq:intermediatestrictclaim}
v_{j}^{R}:=\mathbbm{1}_{\ball(0,R)}(u_{j}-u)\to 0\;\;\;\text{strongly in}\;\lebe^{\frac{n}{n-1}}(\R^{n},V)\;\;\;\text{for all}\;R>0. 
\end{align}
Fix $R>0$. In view the claim, we must establish that (i) $v_{j}^{R}\to 0$ in measure (with respect to  $\mathscr{L}^{n}$) and  (ii) $(v_{j}^{R})$ is $\frac{n}{n-1}$-uniformly integrable. By $\A$-strict convergence, we are in position to use $u_{j}\to u$ in measure (with respect to $\mathscr{L}^{n}$) as a consequence of $u_{j}\to u$ in $\lebe^{1}(\R^{n},V)$. Thus assume toward a contradiction that $(\mathbbm{1}_{\ball(0,R)}u_{j})$ is not $\frac{n}{n-1}$-uniformly integrable. Since $\A[D]$ is elliptic and cancelling, we have 
\begin{align*}
\sup_{j\in\mathbb{N}}\|v_{j}^{R}\|_{\lebe^{\frac{n}{n-1}}(\ball(0,R),V)}\leq \sup_{j\in\mathbb{N}}\|u_{j}-u\|_{\lebe^{1}(\R^{n},V)}+|\A[D]u|(\R^{n})<\infty 
\end{align*}
and similarly $\sup_{j\in\mathbb{N}}|\A v_{j}^{R}|(\ball(0,R))<\infty$. Since $(\mathbbm{1}_{\ball(0,R)}u_{j})$ is assumed to be not $\frac{n}{n-1}$-uniformly integrable, $(v_{j}^{R})$ is not $\frac{n}{n-1}$-uniformly integrable either. By the Banach-Alaoglu-Bourbaki theorem, we thus deduce that there exist two non-negative finite Radon measures $\mu_{1},\mu_{2}\in\mathscr{M}(\overline{\ball(0,R)})$ such that, for a non-relabeled subsequence
\begin{align}\label{eq:measureconvergences1}
|v_{j}^{R}|^{\frac{n}{n-1}}\mathscr{L}^{n}\mres\ball(0,R)\stackrel{*}{\rightharpoonup} \mu_{1}\;\;\;\text{and}\;\;\;|\A[D]v_{j}^{R}|\stackrel{*}{\rightharpoonup} \mu_{2}
\end{align}
as $j\to\infty$; again, we can assume that $\gamma$ is a strictly positive measure as we suppose that $v_{j}^{R}\not\to \mathbbm{1}_{\ball(0,R)}u$ in $\lebe^{\frac{n}{n-1}}(\ball(0,R),V)$ for the time being. Now let $\varphi\in\hold^{1}(\overline{\ball(0,R)})$. Then, employing the Sobolev inequality $\|\psi\|_{\lebe^{\frac{n}{n-1}}(\R^{n},V)}\leq C|\A[D]\psi|(\R^{n})$, we find 
\begin{align}\label{eq:measureconvergences2}
\begin{split}
\int_{\R^{n}}|\varphi u_{j}^{R}|^{\frac{n}{n-1}}\dif x \leq c\Big(\int_{\R^{n}}|\varphi \A[D] u_{j}^{R}| + |\nabla\varphi\otimes_{\A}u_{j}^{R}|(\mathscr{L}^{n}\mres\ball(0,R)) \Big)^{\frac{n}{n-1}},\\
\int_{\R^{n}}|\varphi v_{j}^{R}|^{\frac{n}{n-1}}\dif x \leq c\Big(\int_{\R^{n}}|\varphi \A[D] v_{j}^{R}| + |\nabla\varphi\otimes_{\A}v_{j}^{R}|(\mathscr{L}^{n}\mres\ball(0,R)) \Big)^{\frac{n}{n-1}}.
\end{split}
\end{align}
We again pass to a suitable non-relabeled subsequence, thereby achieving 
\begin{itemize}
\item[(i)] $\varphi u_{j}^{R}\rightharpoonup \varphi u^{R}$ in $\lebe^{\frac{n}{n-1}}(\ball(0,R),V)$. 
\item[(ii)] $\varphi u_{j}^{R}\to \varphi u^{R}$ $\mathscr{L}^{n}$-a.e. in $\ball(0,R)$. 
\end{itemize}
By the Riesz-Fischer theorem, (ii) is easily achievable by exploiting $\varphi u_{j}^{R}\to \varphi u^{R}$ in $\lebe^{1}(\ball(0,R),V)$. Since $(\varphi u_{j}^{R})$ is equibounded in $\lebe^{\frac{n}{n-1}}(\ball(0,R),V)$, we thus deduce by $\varphi u_{j}^{R}\to \varphi u^{R}$ in $\lebe^{1}(\ball(0,R),V)$ that $\varphi u_{j}^{R}\to \varphi u^{R}$ in $\mathscr{D}'(\ball(0,R),V)$ and thus in total\footnote{Here we use that equiboundedness in $\lebe^{p}(\Omega,V)$, $1<p<\infty$ and convergence in $\mathscr{D}'(\Omega,V)$ imply weak convergence in $\lebe^{p}(\Omega)$ provided the underlying set $\Omega$ is bounded.} arrive at (i). We are now in position to use the Brezis-Lieb Lemma ~\ref{lem:BrezisLieb} on the left hand side of \eqref{eq:measureconvergences2}. This yields
\begin{align*}
\int_{\R^{n}}|\varphi u^{R}|^{\frac{n}{n-1}}\dif x & + \int_{\R^{n}}|\varphi|^{\frac{n}{n-1}}\dif\mu_{1}  = \int_{\R^{n}}|\varphi u^{R}|^{\frac{n}{n-1}}\dif x + \lim_{j\to\infty}\int_{\R^{n}}|\varphi|^{\frac{n}{n-1}}| v_{j}^{R}|^{\frac{n}{n-1}}\dif\mathscr{L}^{n} \\ &  = \lim_{j\to\infty} \int_{\R^{n}}|\varphi u_{j}^{R}|^{\frac{n}{n-1}}\dif x \\ 
& \leq c\lim_{j\to\infty}\Big(\int_{\R^{n}}|\varphi \A[D] u_{j}^{R}| + |\nabla\varphi\otimes_{\A}u_{j}^{R}|(\mathscr{L}^{n}\mres\ball(0,R)) \Big)^{\frac{n}{n-1}} \\ 
& \leq c \Big(\int_{\R^{n}}|\varphi|\dif |\A[D]u^{R}|+\int_{\R^{n}}|\nabla\varphi\otimes_{\A}u^{R}|(\mathscr{L}^{n}\mres\ball(0,R)) \Big)^{\frac{n}{n-1}}, 
\end{align*}
whereas $\eqref{eq:measureconvergences2}_{2}$ yields 
\begin{align*}
\int_{\R^{n}}|\varphi|^{\frac{n}{n-1}}\dif\mu_{1} \leq c \Big(\int_{\R^{n}}|\varphi|\dif\mu_{2} \Big)^{\frac{n}{n-1}}.
\end{align*}
By routine approximation, this entails that $\mu_{1}(A)\leq c \mu_{2}(A)^{\frac{n}{n-1}}$ for all $A\in\mathscr{B}(\ball(0,R))$ and thus $\mu_{1}\ll\mu_{2}$. Hence the density $\frac{\dif\mu_{1}}{\dif\mu_{2}}$ is well-defined $\mu_{2}$-a.e., and by $\mu_{1}(A)\leq c \mu_{2}(A)^{\frac{n}{n-1}}$ for all $A\in\mathscr{B}(\ball(0,R))$, we moreover deduce that $\frac{\dif\mu_{1}}{\dif\mu_{2}}(x)=0$ for all $x\in\ball(0,R)$ which are no atoms for $\mu_{2}$. On the other hand, $\mu_{1}$ is finite and absolutely continuous with respect to $\mu_{2}$; thus we find $(a_{l})\in\ell^{1}(\mathbb{N},\R_{\geq 0})$ and distinct points $x_{l}\in\overline{\ball(0,R)}$ with $\mu_{1}=\sum_{l=1}^{\infty}a_{l}\delta_{x_{l}}$. By assumption, $\mu_{1}$ is strictly positive and hence we find $l^{*}\in\mathbb{N}$ with $a_{l^{*}}>0$. We now localise around the point $x_{l^{*}}$ and choose $\varphi\in\hold_{c}^{1}(\ball(0,1),[0,1])$ with $\mathbbm{1}_{\ball(0,\frac{1}{2})}\leq \varphi \leq \mathbbm{1}_{\ball(0,\frac{3}{4})}$ and put, for $\varepsilon>0$, $\varphi_{\varepsilon}(x):=\varphi(\frac{x-x_{l^{*}}}{\varepsilon})$. By the above, we then deduce 
\begin{align*}
\mu_{1}(\{x_{l^{*}}\}) & \leq \lim_{\varepsilon\searrow 0} \int_{\R^{n}}|\varphi_{\varepsilon} u^{R}|^{\frac{n}{n-1}}\dif x  + \int_{\R^{n}}|\varphi|^{\frac{n}{n-1}}\dif\mu_{1}  \\ & \leq c \lim_{\varepsilon\searrow 0}\Big(\int_{\R^{n}}|\varphi_{\varepsilon}|\dif |\A[D]u^{R}|+\int_{\ball(x_{l^{*}},\varepsilon)}|\nabla\varphi_{\varepsilon}\otimes_{\A}u^{R}|(\mathscr{L}^{n}\mres\ball(0,R)) \Big)^{\frac{n}{n-1}} \\ 
& \leq c \Big(|\A[D]u^{R}|(\{x_{l^{*}}\}))^{\frac{n}{n-1}} + c \lim_{\varepsilon\searrow 0}\Big(\|\nabla\varphi_{\varepsilon}\|_{\lebe^{n}(\ball(x_{l^{*}},\varepsilon)}\|u^{R}\|_{\lebe^{\frac{n}{n-1}}(\ball(x_{l^{*}},\varepsilon)} \Big)^{\frac{n}{n-1}} \\
& =  c \Big(|\A[D]u^{R}|(\{x_{l^{*}}\}))^{\frac{n}{n-1}}.
\end{align*}
Here, the ultimate inequality is a consequence of the change of variables
\begin{align*}
\int_{\ball(x_{l^{*}},\varepsilon)}|\nabla\varphi_{\varepsilon}|^{n}\dif x \leq \int_{\ball(0,1)}|\nabla\varphi|^{n}\dif x \leq C 
\end{align*}
and the fact that $u^{R}$ belongs to $\lebe^{\frac{n}{n-1}}(\ball(0,R),V)$. Since $\A[D]$ is elliptic and cancelling, Lemma~\ref{lem:nonatomic} implies that $\A[D]u$ is necessarily non-atomic. Thus $|\A[D]u^{R}|(\{x_{l^{*}}\})=0$, and by the second step, $0<a_{l^{*}}<\mu_{1}(\{x_{l^{*}}\})\leq 0$. This is the desired contradiction, and the proof of \eqref{eq:intermediatestrictclaim} is complete. 

\emph{Step 2. Conclusion of the full claim.} We start by noting that if $\Omega\subset\R^{n}$ is an open set with $|\A[D]u|(\partial\Omega)=0$, then there holds 
\begin{align}\label{eq:strictclaim}
u_{j}\to u\;\;\;\text{$\A$-strictly in $\bv^{\A}(\Omega)$}, 
\end{align}
which may be regarded as \emph{locality} of $\A$-strict convergence. To see \eqref{eq:strictclaim}, we remark that if $\mu,\mu_{1},\mu_{2}...\in\mathscr{M}(\R^{n},W)$ satisfy $\mu_{j}\to\mu$ strictly as $j\to\infty$, then their total variation measures satisfy $|\mu_{j}|\stackrel{*}{\rightharpoonup}|\mu|$. By classical results on weak*-convergence of Radon measures, this implies 
\begin{align*}
|\A[D]u|(\Omega) \leq \liminf_{j\to\infty}|\A[D]u_{j}|(\Omega) \leq \limsup_{j\to\infty}|\A[D]u_{j}|(\overline{\Omega}) \leq |\A[D]u|(\overline{\Omega}). 
\end{align*}
But since $|\A[D]u|(\Omega)=|\A[D]u|(\overline{\Omega})$, the preceding inequality gives $|\A[D]u_{j}|(\Omega)\to|\A[D]u|(\Omega)$ as $j\to\infty$, and together with $u_{j}\to u$ in $\lebe^{\frac{n}{n-1}}(\ball(0,R),V)$, we arrive at \eqref{eq:strictclaim}. We proceed by claiming that
\begin{align}\label{eq:noconcentrationatinfty}
\lim_{R\to\infty}\sup_{j\in\mathbb{N}}|\A[D]u_{j}|(\R^{n}\setminus\ball_{R})=0, 
\end{align}
where $\ball_{R}:=\ball(0,R)$. Indeed, if \eqref{eq:noconcentrationatinfty} were false, then we would find $\theta>0$ such that for each $l\in\mathbb{N}$ there exists $j_{l}\in\mathbb{N}$ with $|\A[D]u_{j_{l}}|(\R^{n}\setminus\ball_{l})\geq \theta$. As $|\A[D]u|$ is a finite Radon measure, we find $R>0$ such that $|\A[D]u|(\ball_{R}^{c})<\frac{\theta}{2}$ and $|\A[D]u|(\partial\!\ball_{R})=0$; the latter is possible by the fact that $|\A[D]u|$ charges at most countably many spheres $\partial\!\ball_{r}$, $r>0$. Since $u_{j_{l}}\to u$ $\A$-strictly in $\bv^{\A}(\R^{n})$, we find by \eqref{eq:strictclaim} that $u_{j}\to u$ $\A$-strictly in $\bv^{\A}(\overline{\ball_{R}}^{c})$. Then, for $l\geq R$,  
\begin{align*}
\theta \leq |\A[D]u_{j_{l}}|(\ball_{l}^{c}) \leq |\A[D]u_{j_{l}}|(\ball_{R}^{c})\stackrel{l\to\infty}{\longrightarrow} |\A[D]u|(\ball_{R}^{c})<\frac{\theta}{2}, 
\end{align*}
an obvious contradiction. To conclude the proof, let $\varepsilon>0$ be given and pick $j_{0}\in\mathbb{N}$ such that $\|u-u_{j}\|_{\lebe^{1}(\R^{n},V)}<\varepsilon$ and $||\A[D]u|(\R^{n})-|\A[D]u_{j}|(\R^{n})|<\varepsilon$ for all $j\geq j_{0}$. Based on \eqref{eq:noconcentrationatinfty}, we find $R_{\varepsilon}>0$ such that $\sup_{j\in\mathbb{N}}|\A[D]u_{j}|(\ball_{R_{\varepsilon}}^{c})<\varepsilon$ and $|\A[D]u|(\ball_{R_{\varepsilon}}^{c})<\varepsilon$. We pick a smooth cut-off $\rho_{R_{\varepsilon}}\in\hold^{\infty}(\R^{n},[0,1])$ with $\mathbbm{1}_{\ball_{R_{\varepsilon}+1}^{c}}\leq \rho_{R_{\varepsilon}} \leq \mathbbm{1}_{\ball_{R_{\varepsilon}}^{c}}$ and $|\nabla\rho_{R_{\varepsilon}}|\leq 1$. By step 1, we find $j_{1}\in\mathbb{N}$ such that $\|u-u_{j}\|_{\lebe^{\frac{n}{n-1}}(\ball_{R_{\varepsilon}+1},V)}<\varepsilon$ for all $j\geq j_{1}$. In consequence, we find for $j\geq \max\{j_{0},j_{1}\}$
\begin{align*}
\|u-u_{j}\|_{\lebe^{\frac{n}{n-1}}(\R^{n},V)} & \leq \|(1-\rho_{R_{\varepsilon}})(u-u_{j})\|_{\lebe^{\frac{n}{n-1}}(\R^{n},V)} + \|\rho_{R_{\varepsilon}}(u-u_{j})\|_{\lebe^{\frac{n}{n-1}}(\R^{n},V)} \\ 
& \leq \|u-u_{j}\|_{\lebe^{\frac{n}{n-1}}(\ball_{R_{\varepsilon}+1},V)} + c\|\A[D](\rho_{R_{\varepsilon}}(u_{j}-u))\|_{\lebe^{1}(\R^{n},W)} \\
& \leq \varepsilon + c\|u_{j}-u\|_{\lebe^{1}(\R^{n},V)} + c|\A[D]u|(\ball_{R_{\varepsilon}}^{c}) + c|\A[D]u_{j}|(\ball_{R_{\varepsilon}}^{c})\\
& \leq c\varepsilon.
\end{align*}
Here, the second inequality follows from $\A[D]$ being elliptic and cancelling, cf. Theorem~\ref{thm:main1} with $\mu=\mathscr{L}^{n}$ and $s=0$. It is clear that $c>0$ merely depends on $\A[D]$, and this completes the proof. 
\end{proof}
We note carefully that the above proof does \emph{not} imply that for elliptic and cancelling operators the embedding $\bv^{\A}(\Omega)\hookrightarrow\lebe^{\frac{n}{n-1}}(\Omega,V)$ is $\A$-strictly continuous. In fact, for mere elliptic and cancelling operators there does not even hold $\bv^{\A}(\Omega)\subset\lebe^{\frac{n}{n-1}}(\Omega,V)$, cf. \cite{GR}. Now we come to the:
\begin{figure}\label{fig:table}
	\begin{tabular}{ c | c | c | c }
	& \multicolumn{3}{ c }{$\A[D]$ is}\\ \cline{2-4}
		$s$ & elliptic  &  elliptic and cancelling &  $\mathbb{C}$-elliptic \\ \hline
		$0\leq s < 1$ & -- & {\small $\A$-strictly conti-} & {\small $\A$-strictly conti-} \\ && {\small nuous into $\lebe^{\frac{n-s}{n-1}}(\dif\mu)$}& {\small nuous into $\lebe^{\frac{n-s}{n-1}}(\dif\mu)$}\\ \hline 
		$s=1$ & -- & -- & {\small norm continuous, not $\A$-strictly} \\ 
 &&& {\small continuous into $\lebe^{1}(\dif\mu)$} \\ \hline 
		$1<s\leq n$ & -- & -- & -- 
	\end{tabular}
	\caption{{\small Context of Propositions~\ref{prop:tracedefine},~\ref{prop:traces=1} and Theorem~\ref{thm:strict}; existence and continuity of codimension $s$ interior trace operators, i.e., for $\mu\in\lebe^{1,n-s}(\R^{n})$, and for first order operators $\A[D]$. For such measures $\mu$, the table lists the continuity properties of the trace embedding operators $\trace_{\Sigma}$ into the corresponding Lebesgue spaces, whereas (--) indicates non-existence of a trace operator. If $s=1$, we suppose that $\mu=\mathscr{H}^{n-1}\mres\Sigma$ for some affine hyperplane or a Lipschitz surface $\Sigma$.}}
	\label{fig:continuity}
\end{figure}
\begin{proof}[Proof of Theorem~\ref{thm:strict}]
Let $u,u_{1},u_{2},...\in\bv^{\A}(\R^{n})$ such that $u_{j}\to u$ $\A$-strictly in $\bv^{\A}(\R^{n})$. By Proposition~\ref{prop:continuitySobolevembedding}, we have that $u_j\rightarrow u$ strongly in $\lebe^{n/(n-1)}(\R^{n},V;\dif\mathscr{L}^n)$. By construction of the trace operator $\trace_{\mu}$ of Proposition~\ref{prop:tracedefine}, we immediately obtain by Theorem~\ref{thm:main1} that, as $\A[D]$ is elliptic and cancelling,  
\begin{align*}
\|\trace_{\mu}(v)\|_{\lebe^{q}(\dif\mu)}\leq c\|\mu\|_{\lebe^{1,n-s}(\R^{n})}^{1/q}\|v\|_{\lebe^{\frac{n}{n-1}}(\R^{n},V)}^{1-\theta}|\A[D]v|(\R^{n})^{\theta}
\end{align*}
for some fixed $s\frac{n-1}{n-s}<\theta<1$ and all $v\in\bv^{\A}(\R^{n})$. In conclusion, applying the preceding inequality to $v=u_{j}-u$, we arrive at
	\begin{align*}
		\|\trace_{\mu}(u_{j}-u)\|_{\lebe^q(\dif\mu)}&\leq c\|\mu\|^{1/q}_{\lebe^{1,n-s}}\|u_j-u\|^{1-\theta}_{\lebe^{\frac{n}{n-1}}(\dif\mathscr{L}^n)}|\A[D]u_j-\A[D]u|(\R^n)^\theta\\
		&\leq c\|\mu\|^{1/q}_{\lebe^{1,n-s}}\|u_j-u\|^{1-\theta}_{\lebe^{\frac{n}{n-1}}(\dif\mathscr{L}^n)}(|\A[D]u_j|(\R^n)+\A[D]u|(\R^n))^\theta\\
		&\rightarrow c\|\mu\|^{1/q}_{\lebe^{1,n-s}}\times 0\times|\A[D]u|(\R^n)^\theta=0,
	\end{align*}
	which concludes the proof.
\end{proof}
The proof moreover demonstrates how the multiplicative trace inequality of Theorem~\ref{thm:main1} crucially boosts norm-continuity of the corresponding trace operators to $\A$-strict continuity. 
\subsection{Examples: $\bv,\bd$ and $\bl$}\label{sec:examplesBVandBD}
We conclude the paper by singling out the implications of the results gathered so far for widely used function spaces. Doing so, we begin with the most classical space, the functions of bounded variation. 
\begin{theorem}[Interior traces and strict continuity for $\bv$]\label{thm:inttracesBV}
Let $n\geq 2$ and $0\leq s \leq 1$. Then, for each $\mu\in\lebe^{1,n-s}(\R^{n})$, there exists a linear trace operator $\trace_{\mu}\colon\bv(\R^{n})\to\lebe^{\frac{n-s}{n-1}}(\R^{n};\dif\mu)$ such that $\trace_{\mu}(\varphi)=\varphi|_{\spt(\mu)}$ for all $\varphi\in(\hold\cap\bv)(\R^{n})$. Moreover, 
\begin{enumerate}
\item\label{item:BV1} if $0\leq s<1$, $\trace_{\mu}$ is continuous for the strict topology, and 
\item\label{item:BV2} if $s=1$, $\trace_{\mu}$ is in general not continuous for the strict topology.
\end{enumerate}
Here, the strict topology is the $\A$-strict topology for the particular choice $\A[D]=D$. 
\end{theorem}

\begin{proof}
The existence of the claimed linear trace operators is a consequence of \cite[Thm.~5.13.1]{Ziemer} for $0\leq s<1$ and \cite[Thm.~5.12.4]{Ziemer} for $s=1$; note that for $0\leq s <1$ this also follows from Proposition~\ref{prop:tracedefine} as $\A[D]$ is elliptic and cancelling. For $0\leq s <1$, the strict continuity assertion of~\ref{item:BV1} is obtained by Theorem~\ref{thm:strict} and so it remains to argue for \ref{item:BV2}. To this end, let $\Sigma:=\partial\!\ball(0,1)$ and $\rho_{j}(x):=\varphi_{j}(|x|)$, $x\in\R^{n}$, where $\varphi_{j}\colon\R\to\R$ is given by 
\begin{align*}
\varphi_{j}(t)=\mathbbm{1}_{[-1-1/j,-1]}(t)(jt+(j+1))+\mathbbm{1}_{(-1,1)}(t)+\mathbbm{1}_{[1,1+1/j]}(t)(-jt+j+1),
\end{align*}
see Figure~\ref{fig:nostrictcontinuity}. Then $\rho_{j}\to \mathbbm{1}_{\ball(0,1)}$ strictly in $\bv(\R^{n})$, $\trace(\rho_{j})=1$ $\mathscr{H}^{n-1}$-a.e. on $\Sigma$ for all $j\in\mathbb{N}$, but we have for the interior trace $\trace_{\Sigma}(1_{\Sigma})=\frac{1}{2}$ $\mathscr{H}^{n-1}$-a.e., cf. \cite[Thm.~3.77, Cor.~3.80]{AFP}. The proof is complete. 
\end{proof}
To the best of our knowledge, \ref{item:BV1} and \ref{item:BV2} of the previous theorem seem to be new. As a consequence of \ref{item:BV2}, we explicitly make the following:
\begin{figure}\label{fig:nostrictcontinuity}
	\hspace{2cm}
	\begin{tikzpicture}[scale=0.8]
	\draw (-1,0) arc (180:360:1cm and 0.25cm);
	\draw (-1,0) arc (180:0:1cm and 0.25cm);
	\draw[dashed,blue] (-1,-1.5) arc (180:360:1cm and 0.25cm);
	\draw[dashed,blue] (-1,-1.5) arc (180:0:1cm and 0.25cm);
	\draw[->] (2.25,-0.5) -- (0.95,-1.25);
	\node at (3.95,-0.45) {\text{$\bv$-trace of $\mathbbm{1}_{\ball(0,1)}$}};
	\draw (-2,-3) arc (180:370:2cm and 0.5cm);
	\draw[dashed] (-2,-3) arc (180:10:2cm and 0.5cm);
	\draw(-2,-2.93)  -- (-1,0);
	\draw(2,-2.93)   -- (1,0);
	\draw[-] (0,0) -- (0,1.25);
	\draw[-] (-0.075,0) -- (0.075,0);
	\draw[-] (1,-3.075) -- (1,-2.925);
	\node at (0.175,0) {\footnotesize{\text{$1$}}};
	\node at (1.1,-3.2) {\footnotesize{\text{$1$}}};
	\draw[dotted,-] (0,0) -- (0,-3);
	\draw[dotted,-] (0,0) -- (0,-3);
	\draw[|-|,red] (-2,-3) -- (-1,-3);
	\draw[dotted] (-2,-3) -- (2,-3);
	\draw[->] (-2.75,-1.85) -- (-1.5,-2.9);
	\node at (-2.885,-1.725) {\text{$\frac{1}{j}$}};
	\draw[-] (-2,-3) -- (-2.5,-3);
	\draw[-] (2,-3) -- (2.5,-3);
	\draw[dotted] (0.75,-2.5) -- (-0.75,-3.5);
	\draw[-] (-0.75,-3.5) -- (-1,-3.68);
	\end{tikzpicture}
	\caption{The situation in the proof of Theorem~\ref{thm:inttracesBV}, depicting the graph of $\rho_{j}$. Then $(\rho_{j})$ converges to $1_{\ball(0,1)}$ strictly and each $\rho_{j}$ has trace $1$ along $\partial\!\ball(0,1)$, but the strict limit has trace $\frac{1}{2}$ $\mathscr{H}^{n-1}$-a.e. on $\partial\ball(0,1)$.}
\end{figure}
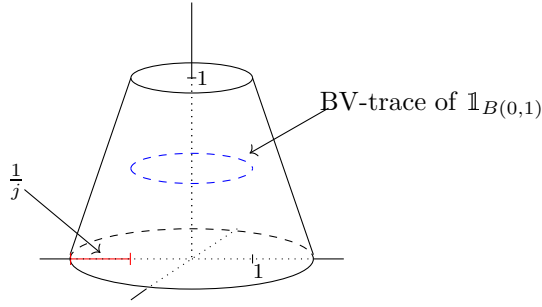
\begin{remark}\label{rem:nonstrictlycontinuous}
Let $\A[D]$ be a first order, $\mathbb{C}$-elliptic operator of the form \eqref{eq:form}. Letting $s=1$ and $\mu=\mathscr{H}^{n-1}\mres\Sigma$ for some hyperplane or sufficiently smooth $(n-1)$-dimensional submanifold of $\R^{n}$,  the trace operator $\trace_{\mu}\colon \bv^{\A}(\R^{n})\to\lebe^{q}(\dif\mu)$ from Proposition~\ref{prop:traces=1} is \emph{not $\A$-strictly continuous in general}. On the contrary, if $0\leq s<1$, then it \emph{is} $\A$-strictly continuous for if $\A[D]$ is elliptic and cancelling by Theorem~\ref{thm:strict}. In this sense, the lower codimension $s<1$ in Theorem~\ref{thm:strict} compensates the lack of $\A$-strict continuity of the interior trace operators.
\end{remark}
If $n\geq 2$, the spaces $\bd(\R^{n})$ and $\bl(\R^{n})$ are obtained from $\bv^{\A}(\R^{n})$ by the particular choice of $\A[D]$ being the symmetric gradient operator (cf.~Example~\ref{ex:symmetricgradient}) or the trace-free symmetric gradient operator (cf.~Example~\ref{ex:tracefreesymmetricgradient}). The symmetric strict and trace-free symmetric strict topologies are defined accordingly for the respective choices of $\A[D]$. Also note that $\bv(\R^{n},\R^{n})\subsetneq\bd(\R^{n})\subsetneq\bl(\R^{n})$. 
\begin{theorem}\label{ex:implications}
Let $0\leq s< 1$. Then the following holds: 
\begin{enumerate}
\item If $n\geq 2$, then for each $\mu\in\lebe^{1,n-s}(\R^{n})$ there exists a linear, symmetric strictly continuous trace operator $\trace_{\mu}\colon\bd(\R^{n})\to\lebe^{\frac{n-s}{n-1}}(\R^{n};\dif\mu)$ such that $\trace_{\mu}(\varphi)=\varphi|_{\spt(\mu)}$ for all $\varphi\in(\hold\cap\bd)(\R^{n})$.
\item If $n\geq 3$ and $\mu\in\lebe^{1,n-s}(\R^{n})$, then there exists a linear, symmetric strictly continuous trace operator $\trace_{\mu}\colon\bl(\R^{n})\to\lebe^{\frac{n-s}{n-1}}(\R^{n};\dif\mu)$ such that $\trace_{\mu}(\varphi)=\varphi|_{\spt(\mu)}$ for all $\varphi\in(\hold\cap\bl)(\R^{n})$.
\end{enumerate}
\end{theorem}	
\begin{proof}
The statement follows from Proposition~Theorem~\ref{thm:strict} since, for the respective choice of underlying dimensions $n$, the operators $\E\,$ and $\E^{D}$ are elliptic and cancelling; note that $\E^{D}$ is \emph{not elliptic and cancelling} for if $n=2$, cf.~Example~\ref{ex:tracefreesymmetricgradient}. 
\end{proof}

\section*{Appendix}
Here we collect some background facts on the strict convergence of measures and discuss its connection to the $\A$-strict convergence of functions of bounded $\A$-variation. Let $(X,d)$ be a metric space and $W$ a finite dimensional real vector space. Namely, we recall from \cite{AFP} that a sequence of finite, $W$-valued Radon measures $(\mu_{j})\subset\mathscr{M}(X,W)$ is said to converge \emph{strictly} to $\mu\in\mathscr{M}(X,W)$ if and only if 
\begin{itemize}
\item[(i)] $\mu_{j}\stackrel{*}{\rightharpoonup}\mu$ and
\item[(ii)] $|\mu_{j}|(X)\to |\mu|(X)$
\end{itemize}
as $j\to\infty$. If $\Omega\subset\R^{n}$ is an open set, the notion of $\A$-strict convergence of a sequence $(u_{j})\subset\bv^{\A}(\Omega)$ to some $u\in\bv^{\A}(\Omega)$ then indeed implies $\A[D] u_{j}\to \A[D]u$ strictly as $j\to\infty$. 

In fact, in this situation the Banach-Alaoglu-Bourbaki theorem that every subsequence $(\A[D]u_{j(l)})$ of $(\A[D]u_{j})$ contains a subsequence $(\A[D]u_{j(l(i))})$ such that, for some $\mu_{j,l}\in\mathscr{M}(\Omega,W)$ there holds $\A[D]u_{j(l(i))}\stackrel{*}{\rightharpoonup} \mu_{j,l}$ as $i\to\infty$. We claim that this weak*-limit is $\A[D]u$ throughout. Namely, let $\varphi\in\hold_{c}^{\infty}(\Omega,W)$ be arbitrary. Then we have 
\begin{align*}
\int_{\Omega}\varphi\dif \A[D]u_{j(l(i))} = - \int_{\Omega}u_{j(l(i))}\A^{*}[D]\varphi\dif x \stackrel{i\to\infty}{\longrightarrow} -\int_{\Omega}u\A^{*}[D]\varphi\dif x = \int_{\Omega}\varphi\dif\A[D]u. 
\end{align*}
Therefore, necessarily $\mu_{j,l}=\A[D]u$, and so we deduce that we also have $\A[D]u_{j}\stackrel{*}{\rightharpoonup}\A[D]u$ in $\mathscr{M}(\Omega,W)$.

\end{document}